\documentclass[12pt]{amsart}
\usepackage{amssymb}
\usepackage{mathtools}
\usepackage[usenames,dvipsnames]{pstricks}
\usepackage{epsfig}
\numberwithin{equation}{section}

\def\v{\varphi}
\def\Re{{\sf Re}\,}
\def\Im{{\sf Im}\,}

\def\1#1{\overline{#1}}
\def\2#1{\widetilde{#1}}
\def\3#1{\widehat{#1}}
\def\4#1{\mathbb{#1}}
\def\5#1{\frak{#1}}
\def\6#1{{\mathcal{#1}}}

\newcommand{\UH}{\mathbb{H}}
\newcommand{\de}{\partial}
\newcommand{\R}{\mathbb R}

\newcommand{\al}{\alpha}

\newcommand{\C}{\mathbb C}

\newcommand{\Hol}{{\sf Hol}}

\newcommand{\D}{\mathbb D}
\newcommand{\oD}{\overline{\mathbb D}}

\newcommand{\N}{\mathbb N}

\def\Re{{\sf Re}\,}
\def\Im{{\sf Im}\,}

\newcommand{\Real}{\mathbb{R}}
\newcommand{\Natural}{\mathbb{N}}
\newcommand{\Complex}{\mathbb{C}}
\newcommand{\ComplexE}{\overline{\mathbb{C}}}

\newcommand{\UD}{\mathbb{D}}

\newcommand{\Step}[2]{\begin{itemize}\item[{\bf Step~#1.}]{\it #2}\end{itemize}}
\newcommand{\step}[2]{\begin{itemize}\item[{\it Step~#1.}]{\it #2}\end{itemize}}
\newcommand{\proofbox}{\hfill$\Box$}

\newcommand{\mcite}[1]{\csname b@#1\endcsname}
\newcommand{\UC}{\mathbb{T}}

\newcommand{\Moeb}{\mathsf{M\ddot ob}}

\newcommand{\diam}{\mathrm{diam}}

\theoremstyle{theorem}

\def\dist{{\rm dist}}

\def\id{{\sf id}}

\def\Re{{\sf Re}\,}
\def\Im{{\sf Im}\,}

\newtheorem{theorem}{Theorem}[section]
\newtheorem{lemma}[theorem]{Lemma}
\newtheorem{proposition}[theorem]{Proposition}
\newtheorem{corollary}[theorem]{Corollary}

\theoremstyle{definition}
\newtheorem{definition}[theorem]{Definition}
\newtheorem{example}[theorem]{Example}

\theoremstyle{remark}
\newtheorem{remark}[theorem]{Remark}
\numberwithin{equation}{section}

\newenvironment{mylist}{\begin{list}{}%
{\labelwidth=2em\leftmargin=\labelwidth\itemsep=.4ex plus.1ex minus.1ex\topsep=.7ex plus.3ex minus.2ex}%
\let\itm=\item\def\item[##1]{\itm[{\rm ##1}]}}{\end{list}}

\newcommand{\clP}{\mathcal P}
\newcommand{\mRe}{\mathop{\mathsf{Re}}}

\newcommand{\anglim}{\angle\lim}

\title[Boundary regular fixed  points]{Boundary regular fixed points in Loewner theory}

\author[F. Bracci]{Filippo Bracci$^\ast$}
\address{F. Bracci, P. Gumenyuk: Dipartimento di Matematica, Universit\`a di Roma ``Tor Vergata", Via della Ricerca
Scientifica 1, 00133, Roma, Italia.} \email{fbracci@mat.uniroma2.it}\email{gumenyuk@mat.uniroma2.it}

\author[M. D. Contreras]{Manuel D. Contreras$^\dag$}

\author[S. D\'{\i}az-Madrigal]{Santiago D\'{\i}az-Madrigal$^\dag$}
\address{M. D. Contreras, S. D\'{\i}az-Madrigal: Camino de los Descubrimientos, s/n\\
Departamento de Matem\'{a}tica Aplicada~II and IMUS\\ Universidad de Sevilla\\ Sevilla,
41092\\ Spain.}\email{contreras@us.es} \email{madrigal@us.es}

\author[P. Gumenyuk]{Pavel Gumenyuk$^\ddag$}

\date{\today}

\thanks{$^\ast\,$Partially supported by the ERC grant ``HEVO - Holomorphic Evolution Equations'' n. 277691.}
\thanks{$^\dag\,$Partially supported by the \textit{Ministerio
de Econom\'{\i}a y Competitividad} and the European Union (FEDER), projects
MTM2009-14694-C02-02 and MTM2012-37436-C02-01, and  by \textit{La Consejer\'{\i}a de Educaci\'{o}n y Ciencia de la Junta de Andaluc\'{\i}a}.}
\thanks{$^\ddag\,$Partially supported by the FIRB grant Futuro in Ricerca ``Geometria Differenziale Complessa e Dinamica Olomorfa'' n. RBFR08B2HY}

\long\def\REM#1{\relax}

\begin{document}

\begin{abstract}
We characterize regular  fixed points  of evolution families in terms of analytical properties of the associated Herglotz vector fields and geometrical properties of the associated Loewner chains. We present several examples showing the r\^ole of  the given conditions. Moreover, we study the relations between evolution families and Herglotz vector fields at regular contact points and prove an embedding result for univalent self-maps of the unit disc with a given boundary regular fixed point into an evolution family with prescribed boundary data.
\end{abstract}

\maketitle

\section{Introduction}

Loewner theory, which originated in Ch.\,Loewner's seminal paper \cite{Loewner} of~1923 and later was developed deeply by P.P.\,Kufarev~\cite{Kuf} and Ch.\,Pommerenke~\cite{Pommerenke-65},\,\cite[Chapter~6]{Pommerenke}, is nowadays one of the main tools in geometric function theory. Loewner Theory proved to be effective in many extremal problems for univalent functions hardly accessible with other methods. The most famous example is its crucial r\^ole in the proof of the Bieberbach conjecture given by L.\,\,de Branges. Recently many
mathematicians have studied a stochastic variant of the Loewner
equation (SLE) introduced by O.\,Schramm. This leads to a breakthrough in several problems of statistical physics and probability theory. A historical overview and bibliography on Loewner Theory can be found, \textit{e.g.}, in survey papers~\cite{ABCD, Br}.

More recently, the authors of this paper have developed a general Loewner theory using an approach, which is different from the classical one and which extends also to complex hyperbolic manifolds \cite{BCM1, BCM2, RMIA, ABHK}. Note that an extension of the classical Loewner theory to several complex variables had been treated for a long time, see, \textit{e.g.},~\cite{Kohr-book}.

According to the new approach, Loewner theory relates three objects: Herglotz vector fields, evolution families and Loewner chains. Roughly speaking, a Herglotz vector field $G(z,t)$ is a
Carath\'eodory vector field such that $G(\cdot, t)$ is
semicomplete for almost every~${t\geq 0}$. An evolution family $(\v_{s,t})$ is a family of
holomorphic self-maps of the unit disc $\D$ satisfying a kind of
semigroup-type algebraic relations and some regularity hypotheses in~$s$ and~$t$. Finally, a Loewner chain $(f_t)$ is
a family of univalent mappings of the unit disc with increasing
ranges satisfying a certain regularity assumption in~$t$. See Section \ref{mmm} for precise definitions and basic results.

This three objects are related by the following fundamental equations:
\[
\frac{\de \v_{s,t}(z)}{\de t}=G(\v_{s,t}(z),t), \quad \frac{\de f_t(z)}{\de t}=-f_t'(z)G(z,t), \quad f_s(z)=f_t(\v_{s,t}(z)).
\]
The aim of the present paper is to study the boundary behavior of the three objects, relating dynamical properties of evolution families with analytical properties of the corresponding Herglotz vector fields and (in some cases) geometrical properties of Loewner chains.

In order to set up our results, we need to introduce some notations and definitions. Following \cite[\S4.3]{P2}, if $f:\UD\to \C$ is holomorphic and $\sigma\in \UC:=\partial\UD$ we say that $f$ is {\sl conformal} at $\sigma$ if  the non-tangential limit of $f$ at $\sigma$ exists---and we denote it  $f(\sigma)$---and the non-tangential limit of the incremental ratio of $f$ at $\sigma$ exists finitely and different from $0$. Let $(\v_{s,t})$ be an evolution family  in $\D$. A point $\sigma\in\UC$ is  a {\sl boundary regular fixed point} of $(\varphi_{s,t})$ if $\varphi_{s,t}$ is conformal at $\sigma$  and $\varphi_{s,t}(\sigma)=\sigma$ for all $t\geq s\ge 0$.
The {\sl spectral function} of $(\varphi_{s,t})$ at a boundary regular fixed point $\sigma\in\UC$ is $\Lambda:[0,+\infty)\to\Real$ defined by $\Lambda(t):=-\log|\varphi_{0,t}'(\sigma)|$. We prove that such a function is of bounded variation.

A {\sl boundary regular null point} for a  holomorphic vector field $H:\UD\to \C$ is a point $\sigma\in \UC$ such that $H$ has a non-tangential singularity at $\sigma$ and the non-tangential limit of the incremental ratio  of $H$ exists finitely at $\sigma$.

Our  main result is the following:

\begin{theorem}\label{main1}
Let $(\varphi_{s,t})$ be an evolution family of order $d\in [1,+\infty]$, let $G$ be its Herglotz vector field, and let $\sigma\in\UC$. The following  assertions are
equivalent:

\begin{mylist}
\item[(A)] for each $t\geq s\geq0$ the point~$\sigma$ is a  boundary regular fixed point of~$(\varphi_{s,t})$;

\item[(B)] The Herglotz vector field~$G$ satisfies the following two conditions:
\begin{itemize}
\item[(B.1)] for a.e. $t\ge0$, $G(\cdot, t)$ has a boundary regular null point at~$\sigma$;
\item[(B.2)] the function $t\mapsto
    G'(\sigma,t)$ is locally integrable in $[0,+\infty)$.
\end{itemize}
\end{mylist}
 Moreover, if one (and hence both) of these assertions holds, then the spectral function $\Lambda$ of $(\varphi_{s,t})$ at~$\sigma$ satisfies
\begin{equation}\label{EQ_G-lambda}
 \Lambda(t)=-\int_0^tG'(\sigma,s)ds\quad\text{for all $t\ge0$.}
\end{equation}
\noindent Furthermore, let $(f_t)$ be a Loewner chain associated with $(\varphi_{s,t})$ and suppose there exists $t_0\ge0$ such that the map $f_{t_0}$ is conformal at  $\sigma$. Then (A) and (B)  are equivalent~to:
\begin{mylist}
\item[(C)] for every $s\ge0$ the following assertions hold:
\begin{itemize}
\item[(C.1)] the map~$f_s$ is conformal at~$\sigma$;
\item[(C.2)] $f_s(\sigma)=f_{t_0}(\sigma)$;
\item[(C.3)] $\displaystyle
    \limsup_{t\to s^+}\big|\arg\big(f'_t(\sigma)/f'_s(\sigma)\big)\big|<\pi$.
\end{itemize}
\end{mylist}
Moreover, if condition (C)  holds, then the function $t\mapsto f_t'(\sigma)$ is locally absolutely continuous on~$[0,+\infty)$, with $\arg f_t'(\sigma)$ being constant.
\end{theorem}
\noindent The proof is given in Section \ref{BRFP-CH}. In Corollary \ref{pole} we show that a similar result holds if $f_{t_0}$ has a simple pole at~$\sigma$ and (C)~is replaced with suitable  conditions for this case. In Section~\ref{EXAMPLE-ES} we present examples showing that conditions~(B.1) and~(C.3) cannot be omitted, and explain the essential role of the conformality of~$f_{t_0}$ at~ $\sigma$. Moreover, several (counter)examples to natural conjectures concerning the regularity of $t\mapsto G'(\sigma,t)$ versus the $L^d$-regularity of the evolution family are also given.

In part, Theorem~\ref{main1} is a consequence of a more general result on regular contact points of evolution families, which relates them with analytic behavior of Herglotz vector fields, see Theorem~\ref{TH_contact}.

Finally, as an application of our main result, in Section \ref{embedding-ch} we prove the following embedding theorem with prescribed boundary data:

\begin{theorem}\label{TH_embedd}
Suppose $\phi\in\Hol(\UD,\UD)$ is univalent and has a boundary regular fixed point at~$\sigma\in\UC$. Then for any $t_0>0$ and for any locally absolutely continuous function $\Lambda:[0,+\infty)\to\Real$ with $\Lambda(0)=0$ and $\Lambda(t_0)=-\log\phi'(\sigma)$ there exists an
evolution family $(\varphi_{s,t})$ satisfying the following conditions:
\begin{mylist}
\item[(i)] $(\varphi_{s,t})$ has a boundary regular fixed point at~$\sigma$,

\item[(ii)] the spectral function of~$(\varphi_{s,t})$ at~$\sigma$ coincides with~$\Lambda$,

\item[(iii)] $\varphi_{0,t_0}=\phi$.
\end{mylist}
\end{theorem}

\section{Preliminaries}

\subsection{Boundary regular contact and fixed points}

In what follows, for a map ${f:\D \to \C}$  and a point $\sigma\in \UC:=\de\D$, we denote by $\angle\lim_{z\to \sigma}f(z)$ the
\textsl{angular} (or \textsl{non-tangential}) \textsl{limit} of $f$ at $\sigma$.

By $\Hol(U,W)$ we will denote the class of all holomorphic maps of~$U$ into~$W$.

\begin{definition}
Let $\varphi\in\Hol(\UD,\UD)$. A point $\sigma\in \UC$ is called a {\sl contact point} if the angular limit
$\varphi(\sigma):=\angle \lim_{z\to \sigma}\varphi(z)$ exists and belongs to~$\UC$. If, in addition, the angular derivative
\begin{equation}\label{EQ_ang-der}
\varphi'(\sigma):=\anglim_{z\to\sigma}\frac{\varphi(z)-\varphi(\sigma)}{z-\sigma}
\end{equation}
exists finitely, then the contact point~$\sigma$ is said to be \textsl{regular}.

A (regular) contact point $\sigma$ for $\v\in\Hol(\UD,\UD)$ is called a \textsl{boundary (regular) fixed point} if $\varphi(\sigma)=\sigma$. For shortness, we will write ``\textsl{BRFP}'' for ``boundary regular fixed point''.
\end{definition}

A characterization of regular contact points is given by the classical Julia\,--\,Wolff\,--\,Carath\'eodory theorem (see, \textit{e.g.}, \cite[\S1.2.1]{Abate} or \cite[p.\,7\,--\,12]{AhlforsConInv}). By the \textsl{boundary dilatation coefficient} of~$\varphi\in\Hol(\UD,\UD)$ at~$\sigma\in\UC$ we mean
\[
    \al_\varphi(\sigma):=\liminf_{z\to \sigma} \frac{1-|\varphi(z)|}{1-|z|}.
\]

\begin{theorem}[Julia\,--\,Wolff\,--\,Carath\'eodory]\label{LM_Julia}
Let $\varphi\in\Hol(\UD,\UD)$ and $\sigma\in\UC$. Then the following four statements are equivalent:
\begin{mylist}
\item[(i)] $\sigma$ is a regular contact point of~$\varphi$;
\item[(ii)]$\alpha_\varphi(\sigma)<+\infty$;
\item[(iii)] there exist a point~$\omega\in\UC$ and $A>0$ such that for all $z\in \UD$
\begin{equation*}
\frac{|\omega-\varphi(z)|^2}{1-|\varphi(z)|^2}\le A \frac{|\sigma-z|^2}{1-|z|^2}
\end{equation*}
\item[(iv)] $\limsup_{(0,1)\ni r\to 1}|\varphi(r\sigma)|=1$ and $\limsup_{(0,1)\ni r\to 1}|\varphi'(r\sigma)|<+\infty$.
\end{mylist}
Moreover, if the above conditions hold, then:
\begin{mylist}
\item[(v)]  the point~$\omega$ in~(iii) is unique and coincides with $\varphi(\sigma):=\anglim_{z\to\sigma}\varphi(z)$;

\item[(vi)] $\varphi'(\sigma)=\angle\lim_{z\to \sigma}\varphi'(z)$ and $\sigma\overline{\omega}\varphi'(\sigma)=\alpha_\varphi(\sigma)=A_0$, where $A_0$ is the least constant~$A$ for which~(iii) holds.
\end{mylist}
\end{theorem}

Denote by $\Moeb(\UD)$ the M\"obius group of all conformal automorphisms of~$\UD$. The classical theorem of Denjoy\,--\,Wolff (see, {\it e.g.} \cite[\S1.3.2]{Abate}) can be stated as follows.

\begin{theorem}[Denjoy\,--\,Wolff]\label{Denjoy-Thm} Let $\varphi\in \Hol(\UD,\UD)$, $\varphi\neq \id_\UD$. Then one of the following alternatives takes place:
\begin{mylist}
  \item[(i)] $\varphi$ is conjugated to a rotation, i.\,e.,  $\varphi\in\Moeb(\UD)$ and has a unique fixed point~$\tau\in\UD$;
  \item[(ii)] the sequence of iterates $(\varphi^{\circ n})$ converges uniformly on compacta to a unique fixed point~$\tau\in\UD$ of $\varphi$;
  \item[(iii)] $\varphi$ has no fixed points in $\UD$, but the sequence of iterates $(\varphi^{\circ n})$ converges uniformly on compacta to a BRFP $\tau\in \UC$ of $\v$ with~$\alpha_\varphi(\tau)\le 1$.
\end{mylist}
Moreover, if $\sigma\neq\tau$ is a contact point of~$\varphi$, then
$\alpha_\varphi(\sigma)\ge \big|1-\overline\tau\varphi(\sigma)\big|^2\big/\big|1-\overline\tau\sigma\big|^2.$ The equality occurs only in case~(i).
\end{theorem}
\noindent The point $\tau$ is called the {\sl Denjoy\,--\,Wolff point} of $\v$ (or, abbreviated, the \textsl{DW-point}).

\subsection{Semigroups and infinitesimal generators}\label{semig} A semigroup $(\phi_t)$ of holomorphic
self-maps of $\D$ is a continuous homomorphism between the
additive semigroup $(\R^+, +)$ of positive real numbers and the
semigroup $({\sf Hol}(\D,\D),\circ)$ of holomorphic self-maps
of $\D$ with respect to the composition, endowed with the
topology of uniform convergence on compacta.

By Berkson\,--\,Porta's theorem \cite{Berkson-Porta}, if $(\phi_t)$
is a semigroup in ${\sf Hol}(\D,\D)$ then $t\mapsto \phi_t(z)$
is analytic and there exists a unique holomorphic vector field
$G:\D\to \C$ such that
\[
\frac{\de \phi_t(z)}{\de
t}=G(\phi_t(z)).
\]
This vector field $G$, called the  {\sl infinitesimal generator} of $(\phi_t)$,  is {\sl semicomplete} in the sense that the Cauchy problem
\[
\begin{cases}
\stackrel{\scriptscriptstyle\bullet}{w}=G(w(t)),\\
w(0)=z,
\end{cases}
\]
has a solution $w^z:[0,+\infty)\to \D$ for any $z\in \D$.
Conversely, any semicomplete holomorphic vector field in $\D$
generates a semigroup in $\Hol(\D,\D)$.

Let $G\not\equiv 0$ be the infinitesimal generator of a one-parameter semigroup $(\phi_t)$. Then there exists a unique
$\tau\in\oD$ and a unique holomorphic $p:\D\to \C$  with $\Re
p(z)\geq 0$ such that the following formula, known as the {\sl Berkson\,--\,Porta
formula}, holds
\[
G(z)=(z-\tau)(\overline{\tau}z-1)p(z).
\]
The point $\tau$ in the Berkson\,--\,Porta formula turns out to be
the DW-point of all $\phi_t$'s different from~$\id_\UD$.
Moreover, if $\tau\in \de\D$, then $\phi_t'(\tau)=e^{\beta t}$ for some $\beta\leq 0$, see \cite[Theorem~(1.7) on~p.\,19]{Siskakis-tesis}.

\begin{definition} A {\sl boundary regular fixed point} of a semigroup
$(\phi_t)$ is a point $\sigma\in \de \D$ which is a boundary regular fixed point of $\phi_t$ for any $t\ge0$.
\end{definition}
\begin{remark}
In fact, the condition in the above definition can be replaced by the weaker assumption that $\sigma$ is a BRFP for \textit{some} $\phi_t\neq\id_\UD$, see~\cite[Theorems 1 and 5]{CDP}.
\end{remark}

\begin{definition}
A {\sl boundary regular null point} (abbreviated, \textsl{BRNP}) of an infinitesimal generator $G$ is a point $\sigma\in \UC:=\de
\D$ such that
\begin{equation}\label{EQ_defBRNP}
G'(\sigma):=\angle\lim_{z\to \sigma}\frac{G(z)}{z-\sigma}
\end{equation}
exists finitely. The number $G'(\sigma)$ is called the {\sl dilation} of $G$ at $\sigma$.
\end{definition}
\begin{remark}
The number~$G'(\sigma)$ in the above definition is always real, see \cite[Theorem~1]{CDP2}.
\end{remark}

In the following theorem we collect some known results concerning infinitesimal generators of
one-parameter semigroups with a BRFP at a given point~$\sigma\in\UC$. By $\clP$ we denote the class of all
$p\in\Hol(\UD,\Complex)$ such that $\Re p(z)\ge0$ for all $z\in\UD$.

\begin{theorem}\label{TH_G-phi}
Let $(\phi_t)$ be a one-parameter semigroup in~$\Hol(\UD,\UD)$ and $G$ its infinitesimal generator. Let $\sigma\in\UC$ and
$\lambda\in\Real$. Then the following statements are equivalent:
\begin{mylist}
\item[(i)] for each $t\ge0$, the function $\phi_t$ has a BRFP at $\sigma$ with $\phi_t'(\sigma)=e^{\lambda t}$;
\item[(ii)] $G$
    has a BRNP at $\sigma$ of dilation $G'(\sigma)=\lambda$;
\item[(iii)] there exits a function $p\in\clP$ such that  $\angle\lim_{z\to
    \sigma}(z-\sigma)p(z)=0$ and
\begin{equation}\label{EQ_BCM-repr}
G(z)=(z-\sigma)(\overline \sigma z-1)\left(p(z)-\frac{\lambda}{2} \frac{\sigma+z}{\sigma-z}\right)\quad \text{ for all $z\in\UD$.}
\end{equation}
\end{mylist}
\end{theorem}
\begin{remark}
The equivalence between (i) and (ii) is from  \cite[Theorem 1]{CDP}, \cite[Theorem 2]{CDP2}, see also \cite{ES}. The Berkson\,--\,Porta type representation at BRNP given by the equivalence between (iii) and (ii) is in \cite{BCD3} (see also \cite{Sh2}). An analogous representation taking into account the position of the DW-point (which is
assumed in this case to be different from the considered BRFP) has been recently given by Goryainov and
Kudryavtseva~\cite{Goryainov-Kudryavtseva}.
\end{remark}

\begin{lemma}\label{LM_G-est} Let $G$ be an infinitesimal generator. Suppose
that $G$ has a BRNP at~$\sigma\in\UC$ with dilation~$\lambda$. Then for all $z\in\UD$,
\begin{equation}\label{EQ_G-est}
|G(z)|\le4\Big(\sqrt2|G(0)|+(\sqrt2+1)\frac{|\lambda|}{2}\Big)\frac{|\sigma-z|^2}{1-|z|^2}.
\end{equation}
\end{lemma}
\begin{proof}
Use representation~\eqref{EQ_BCM-repr} in Theorem~\ref{TH_G-phi}. Then $G(0)=\sigma(p(0)-\lambda/2)$ and hence $|p(0)|\le|G(0)|+|\lambda|/2$.
Furthermore, $|G(z)|\le|\sigma-z|^2\big(|p(z)|+|\frac{\lambda}{2} \frac{\sigma+z}{\sigma-z}|\big)$ for all $z\in\UD$, which being combined with the growth estimates for the class $\clP$ (see, {\it e.g.}
\cite[inequality (11) on~p.\,40]{Pommerenke}) results in
$$
|G(z)|\le|z-\sigma|^2\Big(\sqrt2\big(|G(0)|+|\lambda|/2\big)\frac{1+|z|}{1-|z|}+\frac{|\lambda|}{2}\frac{1+|z|}{1-|z|}\Big).
$$
Taking into account that $\frac{1+|z|}{1-|z|}\le4/(1-|z|^2)$, one easily obtains~\eqref{EQ_G-est}.
\end{proof}

\subsection{Evolution families and generalized Loewner\,--\,Kufarev equation}\label{mmm}

The three main objects of the generalized Loewner theory are Herglotz vector
fields, evolution families and Loewner chains (see \cite{BCM1} and
\cite{RMIA}).

\begin{definition}
\label{DH} Let $d\in [1,+\infty]$. A {\sl Herglotz
 vector field of order $d$} on the unit disc
$\mathbb{D}$ is a function
$G:\mathbb{D}\times\lbrack0,+\infty)\rightarrow \mathbb{C}$
with the following properties:

\begin{enumerate}
\item[H1.] For all $z\in\mathbb{D},$ the function $\lbrack
0,+\infty)\ni t\mapsto G(z,t)$ is measurable;

\item[H2.] For all $t\in\lbrack0,+\infty),$ the function $
\mathbb{D}\ni z\mapsto G(z,t)$ is holomorphic;

\item[H3.] For any compact set $K\subset\mathbb{D}$ and for all $T>0$ there
exists a non-negative function $k_{K,T}\in
L^{d}([0,T],\mathbb{R})$ such that $|G(z,t)|\leq k_{K,T}(t)$
for all $z\in K$ and for almost every $t\in\lbrack0,T].$
\item[H4.] For almost every $t\in [0,+\infty)$,
$G(\cdot, t)$ is an infinitesimal generator.
\end{enumerate}
\end{definition}

In \cite[Theorem 4.8]{BCM1} it is proved that any Herglotz
vector field $G(z,t)$ has an essentially unique representation by means of a
Berkson\,--\,Porta type formula, namely,
$G(z,t)=(z-\tau(t))(\overline{\tau(t)}z-1)p(z,t)$, where
$\tau:[0,+\infty)\to \oD$ is a measurable function and
$p:\D\times [0,+\infty)\to \C$ has the property that for all
$z\in\mathbb{D},$ the function $\lbrack0,+\infty )\ni t \mapsto
p(z,t)\in\mathbb{C}$ belongs to
$L_{loc}^{d}([0,+\infty),\mathbb{C})$; for all
$t\in\lbrack0,+\infty),$ the function $\mathbb{D}\ni z \mapsto
p(z,t)\in\mathbb{C}$ is holomorphic; for all $z\in\mathbb{D}$
and for all $t\in\lbrack0,+\infty),$ we have $\Re p(z,t)\geq0.$

\begin{definition}\label{DE}
Let $d\in [1,+\infty]$. A family $(\varphi_{s,t})_{0\leq s\leq t<+\infty}$ of
holomorphic self-maps of the unit disc  is called an {\sl evolution
family of order $d$} if it satisfies the following conditions:
\begin{enumerate}
\item[EF1.] $\varphi_{s,s}=\id_{\mathbb{D}}$ for any~$s\ge0$;

\item[EF2.] $\varphi_{s,t}=\varphi_{u,t}\circ\varphi_{s,u}$ whenever $0\leq
s\leq u\leq t<+\infty$;

\item[EF3.] for any $z\in\mathbb{D}$ and any $T>0$ there exists a
non-negative function $k_{z,T}\in L^{d}([0,T],\mathbb{R})$
such that
\[
|\varphi_{s,u}(z)-\varphi_{s,t}(z)|\leq\int_{u}^{t}k_{z,T}(\xi)d\xi
\]
whenever $0\leq s\leq u\leq t\leq T.$
\end{enumerate}
\end{definition}

The elements of evolution families are univalent
\cite[Corollary 6.3]{BCM1}.

In \cite[Theorem 1.1, Theorem 6.6]{BCM1} it is proved that
there is a one-to-one correspondence between evolution families
and Herglotz vector fields:

\begin{theorem}\label{LODE}
For any evolution family $(\v_{s,t})$ of order $d\in[1,+\infty]$ there exists a unique
(up to changing on a set of measure zero
 in $t$) Herglotz vector field $G(z,t)$ of order $d$ such
that for all $z\in \D$
\begin{equation}\label{main-eq}
\frac{\de \v_{s,t}(z)}{\de t}=G(\v_{s,t}(z),t) \quad \text{for a.e.
$t\in [0,+\infty)$}.
\end{equation}
Conversely, for any Herglotz vector field $G(z,t)$ of order
$d\in[1,+\infty]$
there exists a unique evolution family $(\v_{s,t})$ of order
$d$  such that \eqref{main-eq} is satisfied.
\end{theorem}

\begin{definition}\label{DL}
Let $d\in[1,+\infty]$. A family $(f_t)_{0\leq t<+\infty}$ of holomorphic maps of the
unit disc is called a {\sl Loewner chain of order $d$}, if the following conditions hold:

\begin{enumerate}
\item[LC1.]  $f_t:\D\to\C$ is univalent for all $t\geq 0$,

\item[LC2.] $f_s(\D)\subset f_t(\D)$ whenever $0\leq s < t<+\infty,$

\item[LC3.] for any compact set $K\subset\mathbb{D}$ and  any $T>0$ there exists a
non-negative function $k_{K,T}\in L^{d}([0,T],\mathbb{R})$
such that
\[
|f_s(z)-f_t(z)|\leq\int_{s}^{t}k_{K,T}(\xi)d\xi
\]
whenever $z\in K$ and $0\leq s\leq t\leq T$.
\end{enumerate}
\end{definition}

\begin{definition}
A Loewner chain~$(f_t)$ is said to be \textsl{associated with} an evolution family~$(\varphi_{s,t})$
if
\begin{equation}\label{EQ_associated}
\varphi_{s,t} = f_t^{-1} \circ f_s\quad\text{for any $t\ge s\ge0$}.
\end{equation}
\end{definition}

In \cite[Theorem 1.3, Theorem 4.1]{RMIA}, see also~\cite[Section~2.1]{SMP_decreasing}, it is
proved

\begin{theorem}\label{pav}Let $d\in[1,+\infty]$. The following statements hold:
\begin{mylist}
\item[\rm(1)] For any Loewner chain $(f_t)$ of order~$d$ the formula~\eqref{EQ_associated} defines an evolution family~$(\varphi_{s,t})$ of the same order~$d$.
Conversely, for any evolution family $(\v_{s,t})$ of
order $d$, there exists a Loewner chain $(f_t)$
of order~$d$ associated with~it.\\[.4ex]
\item[\rm(2)] Let $G(z,t)$ be the Herglotz vector field of an evolution family~$(\varphi _{s,t})$ of order~$d$, and let $(z,t)\mapsto f_{t}(z)$ be a solution to
\begin{equation}\label{low-kuf}
\frac{\partial f_{s}(z)}{\partial s}=-G(z,s)f_{s}^{\prime }(z)
\end{equation}
on $\UD\times[0,+\infty)$.
Suppose that $f_t$ is univalent in~$\UD$ for every~$t\ge0$.
Then $(f_{t})$ is a Loewner chain of order $d$ associated with
the evolution family~$(\varphi _{s,t}).$
\end{mylist}
\end{theorem}

\section{Spectral functions and regular contact  points of evolution families}\label{CONTACT-CH}

\begin{definition}
Let $(\v_{s,t})$ be an evolution family of order $d\in[1,+\infty]$. A point $\sigma\in\UC$ is called a {\sl regular contact point} of $(\varphi_{s,t})$ if  $\sigma$ is a regular contact point of the function $\varphi_{0,t}$ for any $t\ge0$.
The {\sl spectral function} of $(\varphi_{s,t})$ at a regular contact point $\sigma\in\UC$ is $\Lambda:[0,+\infty)\to\Real$ defined by $\Lambda(t):=-\log|\varphi_{0,t}'(\sigma)|$.
\end{definition}
\begin{remark}
From \cite[Lemma~2]{CDP2} it follows that if $\sigma\in\UC$ is a regular contact point of an evolution family
$(\varphi_{s,t})$, then $\sigma(s):=\varphi_{0,s}(\sigma)$ is a regular contact point of the function $\varphi_{s,t}$ for all
$t\geq s\geq 0$.
\end{remark}

\begin{lemma}\label{LM_bounded_var_contact}
Let $(\varphi_{s,t})$ be an  evolution family of order $d\in [1,+\infty]$ in~$\UD$ and let $\sigma\in\UC$ be a regular contact point for $(\v_{s,t})$. Then the spectral function
$\Lambda$ of $(\varphi_{s,t})$ at~$\sigma$ has the following properties:
\begin{mylist}
\item[(i)] $\Lambda(0)=0$;
\item[(ii)] $\big|\varphi'_{s,t}\big(\varphi_{0,s}(\sigma)\big)\big|=e^{\Lambda(s)-\Lambda(t)}$ for any
    $t\geq s\geq 0$;
\item[(iii)] the function $\Lambda$ has locally bounded variation.
\end{mylist}
\end{lemma}
\begin{proof}
Statement~(i) is just by construction, while~(ii) follows from the chain rule for angular derivatives, see,
{\it e.g.}, \cite[Lemma~2]{CDP2}.

Hence we only need to prove~(iii). By (ii) and Theorem~\ref{LM_Julia} applied to $\v=\v_{s,t}$ and $z=0$ with $\varphi_{0,s}(\sigma)$ substituted for~$\sigma$, for any $t\ge s\ge0$ we have
\begin{equation*}
\begin{split}
[\Lambda(t)-\Lambda(s)]^+=\log^+\frac{1}{|\varphi'_{s,t}(\v_{0,s}(\sigma))|}
\leq \log^+\frac{1-|\v_{s,t}(0)|^2}{|\v_{0,t}(\sigma)-\v_{s,t}(0)|^2}
\leq \log\frac{1+|\v_{s,t}(0)|}{1-|\v_{s,t}(0)|}.
\end{split}
\end{equation*}
Combined with condition EF3 from Definition~\ref{DE}, this inequality implies that the total variation of~$\Lambda$ on~$[0,T]$ is finite for any~$T>0$.
\end{proof}

\begin{remark}\label{RM_DW}
In~\cite[Sections\,7,\,8]{BCM1} it has
been proved that in case $\tau\in \UC$ is the DW-point of $\v_{s,t}$ whenever $t\geq s\geq0$ and $\v_{s,t}\neq\id_\UD$, then the spectral function  $\Lambda$ at $\tau$ is absolutely continuous and
\begin{equation*}\label{EQ_Lambda-formula}
\Lambda(t)=-\int_{0}^tG'(\tau,s)\,ds.
\end{equation*}
\end{remark}

Here we study the much more general case of a regular contact point:
\begin{theorem}\label{TH_contact}
Let $(\varphi_{s,t})$ be an evolution family and $G$ its Herglotz vector field.
Suppose $(\varphi_{s,t})$ has a regular contact point~$\sigma\in\UC$ with  spectral function~$\Lambda$. Set $\sigma(t):=\v_{0,t}(\sigma)$. Then the following statements hold:
\begin{mylist}
\item[(i)] for a.e. $t\ge0$ the angular limit $$G(\sigma(t),t):=\angle\lim_{z\to \sigma(t)} G(z,t)$$  exists and $v(t):=-i\,\overline{\sigma(t)}\,G(\sigma(t),t)\in\Real$;
\item[(ii)] the function $[0,+\infty)\ni t\mapsto G(\sigma(t),t)$ is
    of class $L^1_{\rm loc}$;
\item[(iii)] the function $[0,+\infty)\ni t\mapsto \sigma(t)$ is locally absolutely continuous and
    for any $s,t\ge0$, $$\sigma(t)=\sigma(s)+\int_s^t G(\sigma(\xi),\xi)\,d\xi;$$
\item[(iv)] for a.e. $t\ge0$ the angular limit
    $$G'(\sigma(t),t):=\angle\lim_{z\to \sigma(t)}\frac{G(z,t)-G(\sigma(t),t)}{z-\sigma(t)}$$ exists and $\Im G'(\sigma(t),t)=v(t)$;

\item[(v)] the function $[0,+\infty)\ni t\mapsto G'(\sigma(t),t)$ is of class $L^1_{\rm loc}$ and for any $t\ge0$,
$$
\Lambda(t)=\int_0^t\Re G'(\sigma(\xi),\xi)\,d\xi.
$$
\end{mylist}
\end{theorem}

\begin{remark}
Note that by Theorem~\ref{LM_Julia}(vi), $\varphi_{s,t}'(\sigma(s))=\sigma(t)\overline{\sigma(s)}|\varphi_{s,t}'(\sigma(s))|$ for all~$s\ge 0$ and all $t\ge s$. Therefore, statements~(iii)\,--\,(v) of the above theorem implies that
$$
\varphi_{s,t}'(\sigma(s))=\exp\int_s^tG'(\sigma(\xi),\xi)\,d\xi,\qquad\text{for all $t\ge s\ge0$}.
$$
\end{remark}

\begin{proof}[Proof of Theorem~\ref{TH_contact}] We divide the proof in several steps.

\Step1{\it The function $[0,+\infty)\mapsto \sigma(t)$ is of locally bounded variation.} For $t\ge0$ we denote
$$h_t(z):=\frac{z+\zeta(t)}{1+\overline{\zeta(t)} z},$$ where $\zeta(t):=\varphi_{0,t}(0)$.
Note that $h_t$ is an automorphism of $\D$ and $h_t(0)=\zeta(t)$ for all $t\ge0$. For $t\geq s\geq0$ let
$\psi_{s,t}(z):=h_t^{-1}\circ\varphi_{s,t}\circ h_s$. Clearly $\psi_{s,t}\in\Hol(\UD,\UD)$,  $\psi_{s,t}(0)=0$, and
$\psi_{s,t}(b(s))=b(t)$, where $b(t):=h_t^{-1}(\sigma(t))$ for all $t\ge0$. Since the function $t\mapsto\zeta(t)\in\UD$ is locally
absolutely continuous, it is sufficient to prove that $t\mapsto b(t)$ is of locally bounded variation on~$[0,+\infty)$. Indeed,
for any $t\geq s\geq 0$,
\begin{align*}
|\sigma(t)-\sigma(s)|&=|h_t(b(t))-h_s(b(s))|\le |h_t(b(t))-h_t(b(s))|+|h_t(b(s))-h_s(b(s))|\\&\le
\frac{|b(t)-b(s)|}{1-|\zeta(t)|^2}+2\,\frac{|\zeta(t)-\zeta(s)|+|\Im(\zeta(t)\overline{\zeta(s)})|}
{(1-|\zeta(t)|)(1-|\zeta(s)|)}\\&\le \frac{|b(t)-b(s)|}{1-|\zeta(t)|^2}+\frac{4|\zeta(t)-\zeta(s)|}
{(1-|\zeta(t)|)(1-|\zeta(s)|)}.
\end{align*}

Fix now $s\ge0$ and $t\ge s$. Consider  $f\in\Hol(\UD,\UD)$ defined as
$f(z):=b(s)\overline{b(t)}\psi_{s,t}(z)/z$ for all~$z\neq0$ and $f(0):=b(s)\overline{b(t)}\psi_{s,t}'(0).$
By construction, $b(s)$ is a contact point for $f$, $f(b(s))=1$ and the boundary dilation coefficient of $f$ at $b(s)$ is $b(s)\overline{b(t)}\psi_{s,t}'(b(s))-1=|\psi_{s,t}'(b(s))|-1>0$ by Theorem \ref{LM_Julia}\,(vi) and Theorem \ref{Denjoy-Thm}.
Applying Theorem \ref{LM_Julia} to $f$ at $b(s)$ with $z=0$  we immediately get
$$
|1-b(s)\overline{b(t)}\psi_{s,t}'(0)|\le2\,\frac{|1-b(s)\overline{b(t)}\psi_{s,t}'(0)|^2}{1-|\psi_{s,t}'(0)|^2} \le 2 (|\psi_{s,t}'(b(s))|-1).
$$
 Therefore,
\begin{multline}\label{EQ_step1}
|b(t)-b(s)|=|b(t)/b(s)-1|\le|1-\psi_{s,t}'(0)|+|\psi_{s,t}'(0)-b(t)/b(s)|\\
\le|1-\psi_{s,t}'(0)|+2(|\psi_{s,t}'(b(s))|-1).
\end{multline}
Note that by \cite[Lemma~2.8]{RMIA}, $(\psi_{s,t})$ is an evolution family. Furthermore, by construction, the origin is the common
DW-point of~$(\psi_{s,t})$ and $b(0)=\sigma$ is a regular contact point of~$(\psi_{s,t})$. Denote by $\Lambda_0$
the spectral function of~$(\psi_{s,t})$ at~$\sigma$. Then \eqref{EQ_step1}
can be rewritten as
\begin{equation}\label{exp-divido}
    |b(t)-b(s)|\le \big|1-\psi'_{0,t}(0)/\psi'_{0,s}(0)\big|+2\big|1-e^{\Lambda_0(s)-\Lambda_0(t)}\big|\quad\text{for any $t\geq s\geq 0$}.
\end{equation}
Therefore, the statement of Step~1 follows from the fact that by \cite[Theorem~7.1]{BCM1}, ${[0,+\infty)\ni t\mapsto\psi_{0,t}'(0)\neq0}$ is locally absolutely continuous and by
Lemma~\ref{LM_bounded_var_contact}, $\Lambda_0$ is of locally bounded variation on~$[0,+\infty)$.

\Step2{Assertions (i),\,(ii) and (iv) hold.} By Lemma~\ref{LM_bounded_var_contact} and the previous step, $\Lambda$ and $t\mapsto\sigma(t)$ are of locally bounded variation on~$[0,+\infty)$.  Therefore there exists a
null-set $N\subset[0,+\infty)$ such that for any~$t\in[0,+\infty)\setminus N$, the derivatives $\Lambda'(t)$ and $\sigma'(t)$
exist finitely and moreover (see, {\it e.g.}, \cite[Theorem 3.6]{SMP_annulusI})  for any $s\ge0$ and $z\in\UD$ the map $[s,+\infty)\ni t\mapsto\varphi_{s,t}(z)\in\UD$ is differentiable
on~$[s,+\infty)\setminus N$, with $(\partial /\partial t)\varphi_{s,t}(z)=G\big(\varphi_{s,t}(z),t\big)$ for all such~$t$.
Consider the family $(\tilde\varphi_{s,t})$ defined by
$\tilde\varphi_{s,t}(z)=\overline{\sigma(t)}\varphi_{s,t}\big(\sigma(s)z\big)$ for~$t\geq s\geq 0$ and~$z\in\UD$. {\it A
priori} we cannot state that~$(\tilde \varphi_{s,t})$ is an evolution family. However, for all~$s\ge0$ and all $z\in\UD$, the
map $[s,+\infty)\ni t\mapsto\tilde\varphi_{s,t}(z)\in\Hol(\UD,\Complex)$ is differentiable on~$[s,+\infty)\setminus N$, with
$\frac{\partial}{\partial t}\tilde\varphi_{s,t}(z)=\tilde G\big(\tilde\varphi_{s,t}(z),t\big)$ for all~$t\in[s,+\infty)\setminus N$, where we set
\[
\tilde G(z,t):=\overline{\sigma(t)}\big(G(\sigma(t)z,t)-z\sigma'(t)\big).
\]
 Note also that~$\sigma_0=1$ is a BRFP for $\tilde
\varphi_{s,t}$, and by Theorem \ref{LM_Julia}, $\tilde\varphi_{s,t}'(1)=|\varphi_{s,t}'(\sigma(s))|$ for all~$t\geq s\geq 0$.

For any $t_0\in[0,+\infty)\setminus N$, the
semigroup $(\phi_t^{t_0})$ generated by $\tilde G(\cdot,t_0)$ is given by the product formula (see, {\it e.g.}~\cite[Theorem~3]{Shoikhet-trick})
\begin{equation}\label{EQ_ShoikhetTrick}
\phi_t^{t_0}=\lim_{n\to+\infty}\phi_{n,t}^{t_0},\quad\text{where}~ \phi_{n,t}^{t_0}:=\big(\tilde\varphi_{t_0,t_0+t/n}\big)^{\circ n},
\end{equation}
for all $t\ge0$. According to the chain rule for angular derivatives, see,
{\it e.g.}, \cite[Lemma~2]{CDP2} or \cite[Lemma~(1.3.25) in~\S1.3.4]{Abate}, $\phi_{n,t}^{t_0}$ has a BRFP at~$1$ for all $t\geq 0$, $n\in \N$ and
\[
(\phi_{n,t}^{t_0})'(1)=\exp n\big(\Lambda(t_0)-\Lambda(t_0+t/n)\big).
\]
Then $\big(\phi_{n,t}^{t_0}\big)'(1)\to e^{-\Lambda'(t_0)t}$ as $n\to+\infty$. Using Theorem~\ref{LM_Julia}\,(iii) for $\phi_{n,t}^{t_0}$ and passing to the limit as $n\to \infty$, it is not hard to see that $1$ is a BRFP for $\phi_t^{t_0}$ and $\big(\phi_t^{t_0}\big)'(1)\le e^{-\Lambda'(t_0)t}$. Then, by
Theorem~\ref{TH_G-phi}, $\tilde G(\cdot,t_0)$ has a BRNP at~$1$ with dilation  $\tilde G'(1,t_0)\le
-\Lambda'(t_0)$ for all~$t_0\in[0,+\infty)\setminus N$.

On the other hand from \eqref{EQ_BCM-repr} with $\sigma:=1, G:=\tilde G(\cdot, t_0), \lambda:=\tilde G'(1,t_0)$ applied for $z=0$, one easily obtains that  $\tilde G'(1,t_0) \ge -2\mRe \tilde G(0,t_0)$ for all~$t_0\in[0,+\infty)\setminus N$.

Thus we conclude that for any $t\in[0,+\infty)\setminus N$ the infinitesimal generator
$\tilde G(\cdot,t)$ has a BRNP at~$\sigma_0=1$ and
\begin{equation}\label{EQ_Gprime}
-2\Re\overline{\sigma(t)}G(0,t)\le\tilde G'(1,t)\le -\Lambda'(t).
\end{equation}
Recall also that $\sigma(t)\in \UC$ for all $t\in[0,+\infty)\setminus N$. In particular, it follows that assertions~(i) and~(iv) hold, with $G(\sigma(t),t)=\sigma'(t)$,
\begin{equation}\label{vGGG}
v(t)=-i{\sigma'(t)}/{\sigma(t)},\quad\text{and}~G'(\sigma(t),t)=\tilde G'(1,t)+iv(t)\qquad\text{for all
$t\in[0,+\infty)\setminus N$.}
\end{equation}
Finally, assertion~(ii) also holds because $t\mapsto \sigma'(t)$ is locally integrable on~$[0,+\infty)$.

\Step3{Assertion~(iii) holds.}  Recall that the angular derivative coincides, provided it is finite, with the radial limit of the derivative (see, {\it e.g.}, \cite[Prop. 4.7 on p. 79]{Pommerenke}). Therefore, by~\eqref{vGGG}, we have
$$\tilde G'(1,t)+iv(t)~=~G'(\sigma(t),t)~=\lim_{(0,1)\ni x\to1}G'(x\sigma(t),t)\qquad\text{for any $t\in[0,+\infty)\setminus N$.}$$
Since $t\mapsto \sigma(t)$ is continuous, from conditions H1 and H2 in Definition~\ref{DH} it follows that $t\mapsto G'(x\sigma(t),t)$ is measurable on~$[0,+\infty)$ for all~$x\in(0,1)$. Therefore, $t\mapsto G'(\sigma(t),t)$ is measurable too. Now, since $t\mapsto v(t)$ is
locally integrable, \eqref{EQ_Gprime}--\eqref{vGGG} implies that  $t\mapsto G'(\sigma(t),t)$ is also locally integrable
on~$[0,+\infty)$.

Fix now any $t>0$ and write
\begin{equation}\label{EQ_int-eq}
\varphi_{0,t}(x\sigma)=x\sigma+\int_{0}^t G(\varphi_{0,s}(x\sigma),s)\,ds.
\end{equation}
Since the spectral function $\Lambda$ has finite variation, it follows that there exists~$M_t>0$ such that
$|\varphi_{0,s}'(\sigma)|<M_t$ for al~$s\in[0,t]$.  By Theorem~\ref{LM_Julia} applied to $\varphi_{0,s}$ and $z=x\sigma$,
\begin{equation}\label{Hop-ll}
\frac{|\sigma(s)-\varphi_{0,s}(x\sigma)|^2}{1-|\varphi_{0,s}(x\sigma)|^2}\le M_t\frac{1-x}{1+x}\quad\text{for all~$x\in(0,1)$ and
all~$s\in[0,t]$}.
\end{equation}
According to Lebesgue's Dominated Convergence Theorem, combining assertion~(i), Lemma~\ref{LM_G-est} applied to $\tilde G (\cdot,t)$, $t\in[0,+\infty)\setminus N$, inequality~\eqref{Hop-ll}, and the fact that the functions $t\mapsto G'(\sigma(t),t)$, $t\mapsto G(0,t)$, and $t\mapsto
v(t)$ are locally integrable on~$[0,+\infty)$, we obtain~(iii) by passing to
the limit in~\eqref{EQ_int-eq} as $(0,1)\ni x\to 1$.

\Step4{Assertion~(v) holds.} According to the previous step of the proof, $t\mapsto \sigma(t)$ is locally absolutely continuous
on~$[0,+\infty)$. Therefore, by~\cite[Lemma~2.8]{RMIA}, $(\tilde \varphi_{s,t})$ is an evolution family. By construction,  $\tilde G$ is its Herglotz vector field, $\sigma_0=1$ is its RBFP, and $\Lambda$ is its spectral function at~$\sigma_0$. Therefore, bearing in mind~\eqref{vGGG}, in order to prove (v) we may assume that $\sigma(t)\equiv 1$.

From \eqref{EQ_int-eq}, fixing any $t\ge0$ we can write
\begin{equation}\label{EQ_toLimit}
e^{-\Lambda(t)}=\varphi_{0,t}'(1)=\lim_{\substack{x\to1\\x\in(0,1)}}\frac{1-\varphi_{0,t}(x)}{1-x}=1+\lim_{\substack{x\to1\\x\in(0,1)}}\int_0^t\frac{G(\varphi_{0,s}(x),s)}{x-1}\,ds.
\end{equation}

Arguing as at the end of Step 3, use \eqref{Hop-ll} and Lemma~\ref{LM_G-est} to
conclude that the integrand in~\eqref{EQ_toLimit} does not exceed in absolute value $C_t(|G(0,s)|+|G'(1,s)|/2)$ for
all~$s\in[0,t]\setminus N$, all $x\in(0,1)$, and some constant $C_t>0$ not depending on $s$ and~$x$. Recalling that $G(0,\cdot)$ and
$G'(1,\cdot)$ are locally integrable, we  can pass to the
limit  in~\eqref{EQ_toLimit} using Lebesgue's Dominated Convergence Theorem. Hence we obtain
\begin{equation}\label{EQ_Lambda-int-eq}
e^{-\Lambda(t)}=1+\int_0^tG'(1,s)\varphi'_{0,s}(1)\,ds\qquad\text{for all $t\ge0$.}
\end{equation}
 Differentiating~\eqref{EQ_Lambda-int-eq} w.r.t. $t$, we get
$-\Lambda'(t)e^{-\Lambda(t)}=G'(1,t)\varphi'_{0,t}(1)=G'(1,t)e^{-\Lambda(t)}$ for a.e.~$t\ge0$,
from which (v) follows easily. The proof is now complete.
\end{proof}

\section{The proof of Theorem \ref{main1}}\label{BRFP-CH}

In this section we are going to prove Theorem \ref{main1}. We will make use the following lemma, whose proof follows almost literally an argument in the proof of~\cite[Theorem~10.5, p.\,305\,--\,306]{Pommerenke}, so
we omit it.

\begin{lemma}\label{LM_derivative_exists}
Let $g:\UD\to\Complex$ be a univalent holomorphic function, $\sigma\in\UC$, and $\omega\in\Complex\setminus g(\UD)$. Then the
function~$\psi(z):=\big(g(z)-\omega\big)/(z-\sigma)$ is normal in~$\UD$. In particular, if $C$ is a slit in~$\UD$ landing
at~$\sigma$ and the limit~$L:=\lim_{C\ni z\to\sigma}\psi(z)$ exists, then the angular limit of~$\psi$ at~$\sigma$ also exists
and equals~$L$. If in addition $L\neq\infty$, then $\angle\lim_{z\to\sigma} g(z)=\omega$ and the angular derivative~$g'(\sigma)$
of~$g$ at~$\sigma$ exists and equals~$L$.
\end{lemma}

We can now start proving Theorem \ref{main1}. Without loss of generality we assume that~${\sigma=1}$.
Note that if (A) holds, then (B) and \eqref{EQ_G-lambda} follow directly from Theorem \ref{TH_contact}.

Let us show that (B) implies (A). Let ${u(z):=-(1-|z|^2)/|1-z|^2}$ be the (negative) Poisson kernel and define
$$g_z(s,t):=u(\v_{s,t}(z))-e^{\Lambda_0(t)-\Lambda_0(s)}u(z),\quad\Lambda_0(t):=-\int_0^t
G'(1,\xi)d\xi,$$
for $z\in \UD$ and $t\geq s\geq0$.
By Theorem~\ref{LM_Julia} (see also \cite[Proposition~2.3]{BCD2}) the inequality
\begin{equation}\label{EQ_ineq-for-g}
g_z(s,t)\leq 0,\quad \forall z\in\UD~~\forall 0\leq s\leq t,
\end{equation}
is equivalent to  $\varphi_{s,t}$ having a BRFP at~$1$ with
$\varphi_{s,t}'(1)\le e^{\Lambda_0(s)-\Lambda_0(t)}$.

Since by (B) for a.e.~$t\ge0$ the infinitesimal generator $G(\cdot, t)$ has a BRNP of dilation $G'(1,t)$ at $\sigma=1$, it
follows from \cite[Theorem 0.4]{BCD2} that for all $z\in \D$ and a.e.~$t\ge0$,
\begin{equation}\label{EQ_brnp}
\Re\big(v(z)\,G(z,t)\big)+G'(1,t) u(z)\leq 0,\quad \text{where $v:=\partial u/\partial x-i\partial u/\partial y$.}
\end{equation}

Fix now $z\in\UD$ and $s\ge0$. Taking into account that $t\mapsto \varphi_{s,t}(z)$ solves the equation $(\partial/\partial
t)\varphi_{s,t}(z)=G\big(\v_{s,t}(z),t\big)$ on $[s,+\infty)$, we conclude that $t\mapsto g_z(s,t)$ is locally absolutely
continuous on~$[s,+\infty)$ and, with the notation $w(t):=\varphi_{s,t}(z)$,
\[
\frac{\de g_z(s,t)}{\de t}=\Re\Big(v\big(w(t)\big)G\big(w(t),t\big)\Big)+G'(1,t)
u\big(w(t)\big)-G'(1,t)g_z(s,t)\leq -G'(1,t)g_z(s,t)
\]
for a.e. $t\ge s$. Therefore, $h(t):=-(\de/\de t)g_z(s,t)-G'(1,t)g_z(s,t)\geq 0$ for a.e. $t\ge s$. Note that $t\mapsto
g_z(s,t)$ is the solution to the differential equation $(\de/\de t)g_z(s,t)+G'(1,t)g_z(s,t)+h(t)=0$ with the initial condition
$g_z(s,s)=0$. Thus one easily concludes that $g_z(s,t)\leq 0$ for all~$t\ge s$. This proves the implication
(B)~$\Longrightarrow$~(A).

\medskip

The proof of the equivalence between (A) and (C) is divided into several steps.

\Step1{(C.1) and (C.2) imply that for each $t\geq s\geq0$ the point $1$ is a contact point of~$\varphi_{s,t}$. Moreover, if for some $t\geq s\geq0$, $\varphi_{s,t}(1)\neq1$, then $\arg\big(f'_t(1)/f'_s(1)\big)=\pi$.}
Fix any $t\geq s\geq0$. Since by hypothesis $\lim_{r\to1}f_s(r)=f_{t_0}(1)\notin f_t(\D)$, by \cite[Theorem 1, \S II.3]{Goluzin}
the univalence of $f_t$ implies that the limit $\lim_{r\to 1}\varphi_{s,t}(r)=\lim_{r\to 1} \big(f_t^{-1}\circ f_s)(r)$ does exist
and belongs to~$\UC$. Thus $1$ is a contact point of~$\varphi_{s,t}$.

Let us now assume that $\varphi_{s,t}(1)\neq1$. Since $f_t$ is conformal at~$1$, it is also isogonal at this point (here we
follow the terminology from~\cite[\S4.3]{P2}). Therefore (see, e.g., \cite[Theorem 11.6]{P2}) we have:
\begin{mylist}
\item[(a)] for each $\alpha\in(0,\pi/2)$ there exists $\rho_1>0$ such that
$$S(\alpha,\rho_1,\kappa_t):=\big\{w\in\Complex:|\arg\overline{\kappa}_t(w-p)|<\alpha,\,
|w-p|<\rho_1\big\}\subset f_t(\UD),$$ where $\kappa_t:=-f'_t(1)/|f'_t(1)|$ and $p:=f_t(1)=f_{t_0}(1)$;

\item[(b)] for each $\alpha>\pi/2$ there exist \underline{\textit{no}} $\kappa\in\UC$ and $\rho>0$ such that $S(\alpha,\rho,\kappa)\subset f_t(\UD)$.
\end{mylist}
Similarly,
\begin{mylist}
\item[(c)] $S(\alpha,\kappa_s,\rho_2)\subset f_s(\UD)\subset f_t(\UD)$ for all~$\alpha\in(0,\pi/2)$ and some $\rho_2>0$ depending
    on~$\alpha$.
\end{mylist}
\newcommand{\Arg}{\mathop{\mathrm{Arg}}}
It follows that $\theta:=\arg\big(f'_t(1)/f_s'(1)\big)\in\{0,\pi\}$. Indeed, if $\theta\in(-\pi,0)\cup(0,\pi)$, then setting
$\alpha>\max\big\{|\theta|/2,\pi/2-|\theta|/2\big\}$ in~(a) and~(c) one easily concludes that $$S\big(\alpha+|\theta|/2,\kappa_s
e^{i\theta/2},\rho_3\big)\subset f_t(\UD),\quad\text{where~$\rho_3:=\min\{\rho_1,\rho_2\}$},$$ which contradicts~(b). Hence
$\theta\in\{0,\pi\}$

It remains to show that actually $\theta=\pi$. Suppose on the contrary that $\theta=0$. Then the arcs $\gamma_1:=f_t([0,1))$ and
$\gamma_2:=f_s([0,1))$ approach the point~$p=f_t(\UD)$ through the same disc sector~$S(\pi/4,\rho_1,\kappa_t)$, which is a subset
of~$f_t(\UD)$ for $\rho_1>0$ small enough. Hence $\gamma_1$ and $\gamma_2$ are equivalent as slits in~$f_t(\UD)$, \textit{i.e.}
they represent the same accessible boundary point of~$f_t(\UD)$. Then by \cite[Theorem~1, \S II.3]{Goluzin},
$[0,1)=f_t^{-1}(\gamma_1)$ and $\varphi_{s,t}([0,1))=f_t^{-1}(\gamma_2)$ land at the same point on~$\UC$, namely at the
point~$1$. It follows that $\varphi_{s,t}(1)=1$, which contradicts our assumption. Thus $\theta=\pi$.

\Step2{If (C) holds, then (A) holds too and $t\mapsto f'_t(\sigma)$ is locally absolutely continuous on $[0,+\infty)$, with $\arg f'_t(\sigma)$ being constant.}  We are going to prove that $1$ is a regular fixed point for all $\varphi_{s,t}$'s. By Step 1, the point $1$ is a contact point of~$\varphi_{s,t}$ for any~$t\geq s\geq0$. Let us now fix~$s\ge0$ and study the map $$\Phi_s:[s,+\infty)\ni t\mapsto \varphi_{s,t}(1)\in\UC.$$

\step{2.1}{The map $\Phi_s$ is continuous.} Suppose on the contrary that there
 exists $\varepsilon_0>0$, a point $t_0\ge s$ and a convergent sequence $[s,+\infty)\ni t_n\to
 t_0$ such that $|\varphi_{s,t_n}(1)-\varphi_{s,t_0}(1)|>\varepsilon_0$. From
 the fact (see \cite[Proposition~3.5]{BCM1}) that $\varphi_{s,t_n}\to\varphi_{s,t_0}$ locally uniformly in~$\UD$ as
 $n\to+\infty$ it follows that passing if necessary to a subsequence of~$(t_n)$ we may
 assume that $|\varphi_{s,t_n}(r_n)-\varphi_{s,t_0}(1)|<\varepsilon_0/2$, where $r_n:=1-1/n$, for
 all $n\in\Natural$. Now fix any $T>s$ such that $(t_n)\subset [s,T]$. The
sets $C_n:=\varphi_{s,t_n}([r_n,1))\subset \UD$ form a sequence of K\oe be arcs for the sequence of functions $\varphi_{t_n,T}$.
Indeed, on the one hand by construction, $\diam_{\Complex}(C_n)>\varepsilon_0/2$ for all~$n\in\Natural$; while on the other hand,
$\varphi_{t_n,T}(C_n)=\varphi_{s,T}([r_n,1))$ tends, as $n\to+\infty$, to the point $w_0:=\varphi_{s,T}(1)$. By the
Schwarz\,--\,Pick theorem, $|\varphi'_{t_n,T}(z)|(1-|z|^2)\le 1$ for all $z\in\UD$ and all $n\in\Natural$. Hence
by~\cite[Theorem~9.2, p.\,265]{Pommerenke}, $\varphi_{t_n,T}-w_0\to0$ as~$n\to+\infty$. However, by \cite[Proposition~3.5]{BCM1},
$(\varphi_{t_n,T})$ converges to $\varphi_{t_0,T}$ which is univalent in~$\UD$ (see \cite[Corollary~6.3]{BCM1}). This
contradiction proves the statement of Step~2.1.

\step{2.2}{There exists $\varepsilon>0$ such that $\varphi_{s,t}(1)=1$ for all $t\in[s,s+\varepsilon)$.} According to~(C.3) we can
choose $\varepsilon>0$ so that $\big|\arg\big(f_t'(1)/f_s'(1)\big)\big|<\pi$ for all~$t\in[s,s+\varepsilon)$. Applying
Step 1, we easily conclude that $\varphi_{s,t}(1)=1$ for all such~$t$.

\step{2.3}{Fix~$t\ge s$ and suppose that $\varphi_{s,t}(1)=1$. Then there exists the finite angular derivative
$\varphi'_{s,t}(1)=f'_s(1)/f'_t(1)>0$.} By \cite[Proposition 4.13]{P2}, there exists the angular derivative $\varphi'_{s,t}(1)$, which
can be either~$\infty$ or a positive number. Therefore,
$$
\frac{f_t(\varphi_{s,t}(r))-f_t(1)}{\varphi_{s,t}(r)-1}=\frac{f_{s}(r)-f_{s}(1)}{r-1} \frac{r-1}{\varphi_{s,t}(r)-1}\to
f'_{s}(1)\frac{1}{\varphi'_{s,t}(1)}\in \C
$$
as $(0,1)\ni r\to1$. Thus using Lemma~\ref{LM_derivative_exists} for $g:=f_t$, $\sigma:=1$, $\omega:=f_t(1)\not\in f_t(\UD)$, and
$C:=\varphi_{s,t}([0,1))$, we conclude that $\varphi'_{s,t}(1)=f_s'(1)/f_t'(1)\neq\infty$, which proves the statement of
Step~2.3.

\step{2.4}{For any~$t\ge s$, $\varphi_{s,t}(1)=1$.} Suppose this is not the case. Let $t_*:=\inf\{t\ge
s:\varphi_{s,t}(1)\neq1\}$. Then, by Step 2.2, $t_*>s$ and for all $t\in(s,t_*)$ we have $\varphi_{s,t}(1)=1$. Hence by Step~2.1,
$\varphi_{s,t_*}(1)=1$. Furthermore, by Step~2.3 applied to $t:=t_*$, $\varphi_{s,t_*}'(1)\in\Complex$. Therefore,
$\varphi_{s,t}(1)=\varphi_{t_*,t}(1)$ for all $t\ge t_*$. But by Step 2.2 applied with~$t_*$ substituted for~$s$,
$\varphi_{t_*,t}(1)=1$ provided $t-t_*$ is small enough. Thus $t_*<\inf\{t\ge s:\varphi_{s,t}(1)\neq1\}$. This contradiction
proves the statement of~Step~2.4.\vskip3mm

Statements of Step~2.3 and~2.4 imply assertion~(A). Hence~\eqref{EQ_G-lambda} holds and, in particular, $t\mapsto f'_t(1)=f_0'(1)e^{\Lambda(t)}$ is locally absolutely continuous. The proof of Step~2 is complete.

\Step3{(A) implies (C).} For any $s\ge0$ and $t\ge s$, since the point $1$ is a BRFP of the function $\varphi_{s,t}$,  the
angular derivative~$\varphi_{s,t}'(1)$ is a positive number (see, e.g., \cite[Proposition 4.13]{P2}).

Fix $t\ge0$. Assume first that $t>t_0$. Then $f_t(\varphi_{t_0,t}(r))=f_{t_0}(r)$ for all $r\in[0,1)$. In particular,
$f_{t_0}(1)\in\partial f_t(\UD)$, because $\varphi_{t_0,t}(r)\to1$ as $(0,1)\ni r\to1$. Moreover,
$$
\frac{f_t(\varphi_{t_0,t}(r))-f_{t_0}(1)}{\varphi_{t_0,t}(r)-1}=
\frac{f_{t_0}(r)-f_{t_0}(1)}{r-1}\frac{r-1}{\varphi_{t_0,t}(r)-1}\to f'_{t_0}(1)\frac{1}{\varphi'_{t_0,t}(1)}\in \C^*
$$
as $(0,1)\ni r\to1$. Therefore, by Lemma~\ref{LM_derivative_exists} applied to $g:=f_t$, $\sigma:=1$, $\omega:=f_{t_0}(1)\not\in
f_t(\UD)$, and $C:=\varphi_{s,t}([0,1))$, the angular limit $\lim_{z\to1}f_t(z)$ exists and equals~$f_{t_0}(1)$, and the angular
derivative $f'_{t}(1)$ of $f_t$ at~$1$ exists and equals $f_{t_0}'(1)/\varphi_{t_0,t}'(1)\in\Complex^*$.

Now consider the case $t<t_0$. We have $f_t(\UD)\subset f_{t_0}(\UD)\not\ni f_{t_0}(1)$. Note that  the curve $r\in [0,1)\mapsto
\varphi_{t,t_0}([0,1))$ approaches the point $1$ tangentially to the real axis. Therefore,
$$
\frac{f_t(r)-f_{t_0}(1)}{r-1}=\frac{f_{t_0}(\varphi_{t,t_0}(r))-f_{t_0}(1)}{\varphi_{t,t_0}(r)-1}\,\,\frac{\varphi_{t,t_0}(r)-1}{r-1}\to
f'_{t_0}(1)\varphi'_{t,t_0}(1)\in \C^*
$$
as $(0,1)\ni r\to1$. Again using Lemma~\ref{LM_derivative_exists} for $g:=f_t$, $\sigma:=1$, $\omega:=f_{t_0}(1)$, and $C:=[0,1)$, we
conclude that the angular limit $\angle\lim_{z\to1}f_t(z)$ exists and equals~$f_{t_0}(1)$ and that the angular derivative $f'_{t}(1)$ of
$f_t$ at~$1$ exists and equals $f_{t_0}'(1)\varphi_{t,t_0}'(1)\in\Complex^*$.

Note that in both cases $f_t'(1)/f'_{t_0}(1)>0$. Thus the proof is now complete. \proofbox

\section{Embedding into evolution families and inclusion chains}\label{embedding-ch}

In this section we prove Theorem \ref{TH_embedd}. Let us recall the following definition from \cite{CAOT}:

\begin{definition}\label{D_conf-embed}
Let $D_1\subset D_2\subsetneq\Complex$ be two simply connected domains, let $F_j$, for $j=1,2$, be a conformal mapping of~$\UD$ onto
$D_j$ and let $P$ be a prime end of the domain~$D_1$. The domain $D_1$ is said to be {\sl conformally embedded in the
domain~$D_2$ at the prime end~$P$} if there exists a regular contact point $\xi\in\UC$ of the mapping~$\varphi:=F_2^{-1}\circ
    F_1$ such that the point $\omega:=\varphi(\xi)$ corresponds under the mapping~$F_1$ to the prime end~$P$.
\end{definition}

For the definition of primes ends and their correspondence under conformal mappings, see {\it e.g.}~\cite[Chapter~9]{ClusterSets}. Note that the choice of the conformal mappings~$F_j$ in the above definition is clearly irrelevant.

Denote by $r(w_0,D)$ the conformal radius of a simply connected domain~$D$ w.r.t. a point~$w_0\in D$. First we prove the following statement:

\begin{proposition}\label{PR_incl_chains}
Let $(D_t\subsetneq\Complex)_{t\ge0}$ be a family of simply connected domains and ${\Lambda:[0,+\infty)\to\Real}$ a locally
absolutely continuous function with $\Lambda(0)=0$. The following three statements are equivalent:
\begin{mylist}
\item[(I)] The family~$(D_t)$ satisfies the following conditions:
\begin{itemize}
    \item[(Ia)] $D_s\subset D_t$ whenever $0\le s\le t<+\infty$;
    \item[(Ib)] the function $[0,+\infty)\ni t\mapsto r(w_0,D_t)\in(0,+\infty)$ is locally absolutely continuous for some (and hence every) point~$w_0\in D_0$;
    \item[(Ic)] there exists a prime end $P$ of the domain~$D_0$ such that for each $t>0$ the
    domain $D_0$ is embedded in~$D_t$ conformally at the prime end~$P$.
\end{itemize}
\item[(II)] There exists a Loewner chain $(f_t)$ that satisfies the following conditions:
\begin{itemize}
\item[(IIa)] $f_t(\UD)=D_t$ for all~$t\ge0$; \item[(IIb)] the evolution family $\varphi_{s,t}:=f_t^{-1}\circ f_s$, $0\le s\le
    t<+\infty$,
 of~$(f_t)$ has a boundary regular fixed point at~$1$;
 \item[(IIc)] the spectral function of~$(\varphi_{s,t})$ at the point~$1$
 coincides with~$\Lambda$.
\end{itemize}
\item[(III)] There exists a Loewner chain $(f_t)$ that satisfies the above conditions (IIa) and (IIb) (but not
    necessarily~(IIc)).
\end{mylist}
\end{proposition}

\begin{remark}
By \cite[Theorem 1.8 and Remark~1.7]{CAOT}, conditions (Ia) and~(Ib) imply that:
\begin{itemize}
    \item[(KC1)] for any $t>0$, $D_t=\bigcup_{s\in[0,t)} D_s$;
    \item[(KC2)] for any $s\ge0$, $D_s$ is a connected component of the interior of the set $\bigcap_{t>s} D_t$.
\end{itemize}
\end{remark}
\begin{definition}[\cite{CAOT}]
A family~$(D_t\subsetneq\Complex)_{t\ge0}$ of simply connected domains is said to be:
\begin{mylist}
\item[-] an {\sl inclusion chain}, if it satisfies conditions (Ia), (KC1) and (KC2);
\item[-] an {\sl L-admissible family} if it satisfies conditions (Ia) and (Ib);
\item[-] a {\sl chordally admissible family}, if it satisfies conditions (Ia), (Ib) and (Ic).
\end{mylist}
\end{definition}

\begin{remark}
A result similar to the above proposition, \cite[Theorem~4.8]{CAOT}, was earlier proved for the case of the common boundary Denjoy\,--\,Wolff point. In Proposition~\ref{PR_incl_chains} we consider the more general case of common BRFPs, but even so most of the arguments from~\cite{CAOT} can be still used. For example, the equivalence between (II) and~(III) in Proposition~\ref{PR_incl_chains} is essentially the same as the equivalence between (c) and~(d) in \cite[Theorem~4.8]{CAOT}. The substantial improvement, made by combining methods from~\cite{CAOT} with Theorem~\ref{TH_contact}, resides in the fact that, according to Proposition~\ref{PR_incl_chains}, for any chordally admissible family~$(D_t)$ there exists a Loewner chain $(f_t)$ of chordal type \textit{such that ${f_t(\UD)=D_t}$ for all~$t\ge0$}, while \cite[Theorem~4.8]{CAOT} guarantee \textit{only the equality of sets~${\{f_t(\UD):t\ge0\}}$ and~${\{D_t:t\ge0\}}$}. The necessity to use Theorem~\ref{TH_contact} in the proof explains why the parameter~$d$ (the order of the evolution family $(\varphi_{s,t})$) does not appear in Proposition~\ref{PR_incl_chains}.
\end{remark}

\begin{proof}[\textbf{Proof of Proposition~\ref{PR_incl_chains}.}] The proof is
divided into three steps.

\Step1{(II)$\Rightarrow$(III)$\Rightarrow$(I).} The fact that (II) implies (III) is completely trivial. Assume that~(III) takes place. Then by \cite[Theorem~2.3]{CAOT}, conditions (Ia) and (Ib) hold, \textit{i.e.} $(D_t)$ is an L-admissible family. Now applying \cite[Lemma~4.7]{CAOT} to $(F_t):=(f_t)$ we see that $(D_t)$ is in fact chordally admissible, \textit{i.e.}, (Ic) holds as well.

\Step2{(I) implies (III), with the evolution family~$(\varphi_{s,t})$ satisfying~$\varphi_{s,t}'(1)=1$.}
By \cite[Lemma~4.7]{CAOT} there exists a family of conformal maps~$(F_t:\UD\to\Complex)_{t\ge0}$ such that
$F_t(\UD)=D_t$ for all~$t\ge0$ and that the function $\psi_t:=F_t^{-1}\circ F_0$ has a BRFP at~$1$. At the same time, by
\cite[Theorem~2.3]{CAOT} there exists a Loewner chain~$(g_t)$ such that $g_t(\UD)=\Omega_t$ for all~$t\ge0$.

Since $g_t(\UD)=F_t(\UD)$, it follows that for any $t\ge0$ there exists $h_t\in\Moeb(\UD)$ such that $F_t=g_t\circ h_t$ and hence
$\phi_{0,t}=h_t\circ\psi_t\circ h_0^{-1}$ for all $t\ge0$, where $(\phi_{s,t})$ is the evolution family of the Loewner
chain~$(g_t)$. It follows that $\sigma_0:=h_0(1)$ is a regular contact point of $(\phi_{s,t})$ and that
$h_t(1)=\phi_{0,t}(\sigma_0)$, $|h_t'(1)|\exp \Lambda_0(t)=|h_0'(1)|$ for all~$t\ge0$, where $\Lambda_0$ is the spectral function
of~$(\phi_{s,t})$ at~$\sigma_0$. Therefore, there exists a family $(\ell_t)_{t\ge0}\subset\Moeb(\UD)$ of automorphisms with
parabolic  fixed point at~$1$ such that
$$\tilde h_t:=h_t\circ \ell_t=\big(z\mapsto
\phi_{0,t}(\sigma_0)\overline\sigma_0 z\big)\circ h_0\circ \left(z\mapsto \frac{z+x(t)}{1+x(t)z}\right),$$ where
$x(t):=(e^{\Lambda_0(t)}-1)/(e^{\Lambda_0(t)}+1)$. Note that~$\ell_0=\id_\UD$.

By Theorem~\ref{TH_contact}, the functions~$\Lambda_0$ and $t\mapsto\phi_{0,t}(\sigma_0)$ are locally absolutely continuous.
Therefore, by \cite[Lemma 2.8]{RMIA}, the formula $\varphi_{s,t}=\tilde h_t^{-1}\circ\phi_{s,t}\circ \tilde h_s$,
$t\geq s\geq0$, defines an evolution family~$(\varphi_{s,t})$. Note that the univalent functions $f_t:= F_t\circ\ell_t=g_t\circ
\tilde h_t$ satisfy $f_t\circ\varphi_{s,t}=f_s$ for any~$t\geq s\geq0$. Thus $(f_t)$ is a Loewner chain associated
with~$(\varphi_{s,t})$, see \cite[Lemma~3.2]{RMIA}. Finally, note that
$\varphi_{0,t}=\ell_t^{-1}\circ\psi_t$ has a parabolic fixed point at~$1$.

\Step3{(I)$\Rightarrow$(II).}  We are going to modify the Loewner chain~$(f_t)$ constructed in the proof of
Step 2 in such a way that the spectral function of the evolution family of this new Loewner chain
at~$1$ coincides with~$\Lambda$. To this end we define $m_t(z):=(z+x(t))/(1+x(t)z)$ for all $z\in\UD$ and $t\ge0$, where
$x(t):=(e^{\Lambda(t)}-1)/(e^{\Lambda(t)}+1)$. Then $\psi_{s,t}:=m_t\circ\varphi_{s,t}\circ m_s^{-1}$ satisfies $\psi_{s,t}(1)=1$
and $\psi_{s,t}'(1)=\exp\big(\Lambda(s)-\Lambda(t)\big)$ for all $t\geq s\geq0$. Recall that $\Lambda$ is locally absolutely
continuous by the hypothesis. Therefore, applying \cite[Lemmas~2.8 and~3.2]{RMIA} in the same way as in the proof of
Step 2, we see that the functions $g_t=f_t\circ m_t^{-1}$ form a Loewner chain whose evolution family
is~$(\psi_{s,t})$. This is the desired Loewner chain and the proof is now complete.
\end{proof}

\begin{remark}\label{RM_incl_chains}
The Loewner chains~$(f_t)$ we construct in the proof of Proposition~\ref{PR_incl_chains} are such that the point~$1$
corresponds under the conformal map~$f_0:\UD\to D_0$ to the prime end~$P$ from condition~(Ic).
\end{remark}

As a preparation to the proof of~Theorem~\ref{TH_embedd}, we need the following lemmata.
\begin{lemma}\label{LM_boundedLCh}
Let $G\subsetneq\UD$ be a simply connected domain, $w_0\in G$, and $t_0>0$. Then there exists an L-admissible family of domains~$(D_t)$ such
that $D_0=G$, $D_{t_0}=\UD$, and $D_t\subset\{z:\Re z<1\}$ for all $t\ge0$.
\end{lemma}
\begin{proof}
First of all without loss of generality we may assume that $t_0=1$. Furthermore, using M\"obius transformations we may assume that
$w_0=0$. Indeed, if $\ell\in\Moeb(\UD)$ and $(D_t)$ is an L-admissible family such that $D_1=\UD$, then $(\tilde D_t)$ defined by
$\tilde D_t:=\ell(D_t)$ for $t\in[0,1]$ and $\tilde D_t:=D_t$ for all $t\ge1$, is again an L-admissible family.

Denote $M:=1/r(G,0)>1$, and let $\phi$ be the conformal map of~$\UD$ onto $G$ normalized by $\phi(0)=0$ and $\phi'(0)>0$. Then the
function $f_0=M\varphi$ belongs to the class $$\mathcal
S^M:=\Big\{f(z)=z+\textstyle{\sum_{n=2}^{+\infty}a_nz^n}:\,|f(z)|<M\text{~for all~$z\in\UD$ and $f$ is univalent in~$\UD$}\Big\}.
$$
It is known (see, {\it e.g.}, \cite[p.\,69\,--\,70]{Aleksandrov}) that for any $f_0\in \mathcal S^M$ there exists a function
${p:\UD\times[0,T]\to\Complex}$, where $T:=\log M$, such that
\begin{mylist}
\item[(i)] for any $t\in[0,T]$ the function~$p(\cdot,t)$ belongs to the Carath\'eodory class, \textit{i.e.} it is a holomorphic function in~$\UD$ with positive real part and normalized by~$p(0,t)=1$;

\item[(ii)] for each
    $z\in\UD$ the function~$p(z,\cdot)$ is measurable in~$[0,T]$;

\item[(iii)] for all $z\in\UD$ we have $f_0(z)=M\psi_{z,0}(T)$, where $w(t)=\psi_{z,s}(t)$, $s\in[0,T]$, $z\in\UD$,
    stands for the unique solution to the initial value problem
$$
\frac{dw(t)}{dt}=-w(t)p\big(w(t),t\big),\quad w(s)=z.
$$
\end{mylist}
For each $s,t\in[0,T]$ with $s\le t$ the function $\varphi_{s,t}(z):=\psi_{z,s}(t)$ is a well-defined univalent holomorphic
self-map of $\UD$ with $\varphi_{s,t}(0)=0$, $\varphi_{s,t}'(0)=e^{s-t}$ and $\varphi_{u,t}\circ\varphi_{s,u}=\varphi_{s,t}$ whenever
$0\le s\le u\le t\le T$, see, \textit{e.g.}, \cite[proof of~Theorem 6.3, p.\,160\,--\,162]{Pommerenke}. Note also that by
construction, $\varphi_{0,T}=\phi$. In particular, it follows that the family of simply connected domains
$D_t:=\varphi_{(1-t)T,\,T}(\UD)$ for all $t\in[0,1]$ and $D_t:=t\UD+1-t$ for all $t\ge1$, satisfies conditions $D_0=G$, $D_1=\UD$,
$D_s\subset D_t$ whenever $0\le s\le t\le 1$, and that $0\le t\mapsto r(D_t,0)$ is a piecewise smooth function. Then by
the very definition, $(D_t)$ is an L-admissible family. Finally, by construction, $D_t\subset\{z:\Re z<1\}$ for all $t\ge0$. The proof is
now complete.
\end{proof}

\begin{remark}
An alternative proof of the above lemma, not using representation of the class $\mathcal S^M$, can be obtained by a modification
of the proof of \cite[Theorem~6.1 on p.\,159]{Pommerenke}.
\end{remark}

\begin{lemma}\label{LM_conformal-embedding}
Let $D_0\subset D_\infty\subsetneq\Complex$ be two simply connected domains. If $D_0$ is embedded in~$D_\infty$ conformally at
some prime end~$P$, then it is also conformally embedded at $P$ in any simply connected domain~$D$ satisfying~$D_0\subset D\subset
D_\infty$.
\end{lemma}
\begin{proof}
By the very definition, the property of being embedded conformally at a prime end is conformally invariant. Therefore we may
assume that~$D_\infty=\UD$. Let $\phi_0$ and $\phi$ be any conformal mappings of~$\UD$ onto~$D_0$ and $D$, respectively. By the
definition of the conformal embedding at a prime end, the point~$\sigma_0\in\UC$ that corresponds to the prime end~$P$
under~$\phi_0$ is a regular contact point of~$\phi_0$. Denote by~$\gamma$ the arc~$\phi_0([0,\sigma_0))$. Since by hypothesis
$D_0\subset D\subset\UD$, this arc is a slit in~$D$. Therefore, by Koebe's Theorem (see, \textit{e.g.}, \cite[Theorem 1, \S II.3]{Goluzin} or \cite[Lemma~2 on p.\,162]{Shapiro-book}) the
arc~$C:=\phi^{-1}(\gamma)$ lands on~$\UC$ at some point~$\sigma\in\UC$. Consider the function $\varphi:=\phi^{-1}\circ\phi_0$.
Clearly, $\varphi$ maps $\UD$ conformally into itself. Moreover, since by construction $\phi(C)=\gamma$ and
$\varphi([0,\sigma_0))=C$ and since the angular derivative $\phi_0'(\sigma_0)$ exists finitely and does not vanish,
\begin{align}\label{EQ_contact-point-varphi}
&\lim_{r\to1^-}\varphi(r\sigma_0)=\sigma,\quad\text{and}\\ \label{EQ_ang-derivative}
&\lim_{r\to1^-}\frac{\varphi(r\sigma_0)-\sigma}{r\sigma_0-\sigma_0}=\phi_0'(\sigma_0)\left(\lim_{C\ni z\to
\sigma}\frac{\phi(z)-\phi_0(\sigma_0)}{z-\sigma}\right)^{\!\!-1}\!\!\!\!\!.
\end{align}
By~\eqref{EQ_contact-point-varphi}, $\varphi$ has a contact point at~$\sigma_0$. It follows, see \textit{e.g.} \cite[Proposition
4.13]{P2}, that the limit~\eqref{EQ_ang-derivative} exists and belongs to~$\ComplexE\setminus\{0\}$. Hence the limit in the
right-hand side of~\eqref{EQ_ang-derivative} also exists and must be finite. By Lemma~\ref{LM_derivative_exists} combined
with~\cite[Proposition 4.13]{P2}, it is, moreover, different from zero. Thus the limit~\eqref{EQ_ang-derivative} is finite,
\textit{i.e.}, $\varphi$ has a \textsl{regular} contact point at~$\sigma_0$. This proves the lemma.
\end{proof}

\begin{proof}[\textbf{Proof of Theorem~\ref{TH_embedd}.}] Clearly, we may assume that $\sigma=1$.
Fix any point $w_0\in\phi(\UD)$. By Lemma~\ref{LM_boundedLCh} there exists an L-admissible family~$(D_t)$ such that $D_0=\phi(\UD)$, $D_{t_0}=\UD$, $D_t\subset D_\infty:=\{z:\Re z<1\}$ for all~$t\ge0$.

Note that, $(D_t)$ is in fact chordally admissible. Indeed, let $P$ be the prime end of~$D_0$ that corresponds to~$1$ under the mapping~$\phi$. The function $\varphi:=\phi_\infty^{-1}\circ\phi$, where $\phi_\infty(z):=2z/(1+z)$ maps conformally~$\UD$ onto~$D_\infty$, has a BRFP at~$1$, because $\phi$  has a BRFP at~$1$ by hypothesis and $\phi_\infty^{-1}$ has also a fixed point at~$1$, being holomorphic there. Then by Lemma~\ref{LM_conformal-embedding}, $D_0$ is embedded in all $D_t$'s conformally at~$P$, which proves our claim.

Hence by Proposition~\ref{PR_incl_chains} and Remark~\ref{PR_incl_chains}, there exists a Loewner chain $(f_t)$ such that:
\begin{itemize}
\item[(a)] conditions (IIa), (IIb), and (IIc) from Proposition~\ref{PR_incl_chains} are fulfilled; \item[(b)] the
    point~$1$ corresponds under the map~$f_0:\UD\to D_0$ to the prime end~$P$.
\end{itemize}
In particular, $f_{t_0}\in\Moeb(\UD)$. It follows that, replacing~$(f_t)$ if necessary with the Loewner chain $(f_t\circ
f_{t_0}^{-1})_{t\ge0}$, we may assume that:
\begin{itemize}
\item[(c)] $f_{t_0}=\id_\UD$; \item[(d)] $f_0(\UD)=\phi(\UD)\subset\UD$, and $\varphi_{0,t_0}=f_0$ has a BRFP at~$1$,
    with $\varphi_{0,t_0}'(1)=\phi'(1)$, where $(\varphi_{s,t})$ stands for the evolution family of the Loewner chain~$(f_t)$.
\end{itemize}

Therefore, there exists $\ell\in\Moeb(\UD)$ of the form $\ell=p_0^{-1}\circ(w\mapsto w+iv_0)\circ p_0$, where
$p_0(z):=(1+z)/(1-z)$ is the Cayley map of~$\UD$ onto~$\UH:=\{w:\Re w>0\}$ sending~$1$ to~$\infty$, such that
$\phi=\varphi_{0,t_0}\circ\ell$. Define now $\ell_t:=p_0^{-1}\circ(w\mapsto w+itv_0/t_0)\circ p_0$ for all~$t\ge0$. Then, by
\cite[Lemma~2.8]{RMIA} the family~$(\psi_{s,t})$ defined by $\psi_{s,t}:=\ell^{-1}\circ \ell_t\circ \varphi_{s,t}\circ
\ell_t^{-1}\circ\ell$ for all $t\geq s\geq0$, is an evolution family in~$\UD$. Note that $\ell_0=\id_\UD$ and $\ell_{t_0}=\ell$
by construction. Thus $\psi_{0,t_0}=\varphi_{0,t_0}\circ\ell=\phi$. The proof is now complete.
\end{proof}

\section{Examples and Remarks on Theorem \ref{main1}}\label{EXAMPLE-ES}

Theorem \ref{main1} contains a characterization of boundary regular fixed points  of an evolution family in terms of the associated Loewner chain assuming that one of the functions of
the Loewner chain is conformal at a certain boundary point. A similar result can be obtained when the function has a simple pole.
Let us recall that a function $f:\UD\to\C$ has a {\sl simple pole (in the angular sense) at a boundary point $\sigma\in \mathbb
T$} if $\angle \lim_{z\to \sigma} (z-\sigma)f(z)\in \C^*.$

Following the notation in the classical case, we call this limit the {\sl residue of $f$ at $\sigma$} and denote it by $
\mathrm{Res}  (f;\sigma)$.
\begin{corollary}\label{pole}
Let $(f_t)$ be a Loewner chain and $(\varphi_{s,t})$ its evolution family. Suppose there exists $t_0\ge0$ such that the map
$f_{t_0}$ has a simple pole at a point $\sigma\in\UC$. Then the following two conditions are equivalent:
\begin{mylist}
\item[(i)] for every $s\ge0$ the following assertions hold:
\begin{itemize}
\item[(i.1)] the map~$f_s$ has a simple pole at the point~$\sigma$; \item[(i.2)] $\displaystyle \limsup_{t\to
    s^+}\big|\arg\big(\mathrm{Res}  (f_s;\sigma)/\mathrm{Res}  (f_t;\sigma)\big)\big|<\pi$.
\end{itemize}
\item[(ii)] the evolution family~$(\varphi_{s,t})$ has a boundary regular  fixed point at~$\sigma$.
\end{mylist}

Moreover, if conditions (i) and (ii) above hold, then
\begin{mylist}
\item[(iii)] the function $t\mapsto \mathrm{Res}  (f_t;\sigma)$ is locally absolutely continuous on~$[0,+\infty)$, with $\arg
    \big(\mathrm{Res}  (f_t;\sigma)\big)$ being constant.
\end{mylist}
\end{corollary}

\begin{proof}
Let us fix $T>t_0$ and a point $w_0\notin f_T(\UD)$.  Write $l(t):=Tt/(1+t)$ and consider the family of univalent functions
$g_t:\UD\to \C$ given by $z \mapsto g_t(z):=1/(f_{l(t)}(z)-w_0)$, for all $t\geq 0$ and $z\in \UD$. It is clear that $f_{l(t)}$
has a simple pole at~$\sigma$ if and only if $g_t(\sigma)=0 $ and $g_t$ is conformal at~$\sigma$. In such a case,
$g'_t(\sigma)=1/\mathrm{Res}  (f_t;\sigma)$. On the other hand, the evolution family associated with the Loewner chain $(g_t)$ is
nothing but $(\varphi_{l(s),l(t)})$. Bearing in mind these remarks and applying Theorem \ref{main1} to the Loewner
chain $(g_t)$ one can easily complete the proof.
\end{proof}

\begin{remark}
Thanks to the chain rule for angular derivatives, see,
{\it e.g.}, \cite[Lemma~2]{CDP2}, and condition~EF2 in the definition of evolution families, hypothesis (A) in Theorem \ref{main1} is equivalent to
\begin{mylist}
\item[(A${\vphantom{A}}'$)]{\it for each $t\ge0$ the point~$\sigma$ is a BRFP of~$(\varphi_{0,t})$.}
\end{mylist}
\end{remark}

Now we consider a family of examples demonstrating that conditions and conclusions in Theorem~\ref{main1} are the best
possible in some sense.

It easily follows from the definition that if $G$ is a Herglotz vector field, then $t\mapsto G'(z,t)$ is locally integrable
on~$[0,+\infty)$ for any~$z\in\UD$. This property is not inherited when passing to the boundary:

\begin{example}\label{EX_non-int-dilation}
We construct an
example of a Herglotz vector field having a BRNP at~$1$ for a.e.~$t\ge0$ such that the dilation~$\lambda(t):=G'(1,t)$
is not locally integrable. We actually show that any non-negative measurable function $\lambda$ defined a.e. on~$[0,+\infty)$ can
arise in this way.

Let $r$ and $\lambda$ be measurable non-negative functions defined a.e. on~$[0,+\infty)$. Assume that $r(t)<1$ for a.e.~$t\ge0$
and that $M(t):=\lambda(t)\big(1-r(t)\big)$ is locally integrable on~$[0,+\infty)$. Clearly, given any non-negative measurable
function~$\lambda$ one can find a function~$r$ such that these requirements are met.

Denote by $p_0(z):=(1+z)/(1-z)$ and consider the function
$$
G_{\lambda,r}(z,t):=(1-z)^2\frac{\lambda(t)}{2}\big(p_0(r(t)z)-p_0(z)\big),
$$
for a.e. $t\ge0$ and all $z\in\UD$. By Theorem~\ref{TH_G-phi}, for a.e.~$t\ge0$, $G_{\lambda,r}(\cdot,t)$ is an infinitesimal
generator with a BRNP of dilation~$\lambda(t)$ at $1$. Moreover, a simple calculation shows that
\begin{equation}
p_0(rz)-p_0(z)=-\frac{2(1-r)z}{(1-z)(1-rz)}.
\end{equation}
Therefore, $|G_{\lambda,r}(z,t)|\le 2M(t)$ for a.e.~$t\ge0$.

Thus, $G_{\lambda,r}$ is a Herglotz vector field having, for a.e.~$t\ge0$, a BRNP of dilation~$\lambda(t)$ at~$1$. Since
we can choose any measurable non-negative function for~$\lambda$, this example shows that condition~(B.2) in
Theorem~\ref{main1} does not follow from condition~(B.1).
\end{example}

\begin{remark}\label{RM_caseDW}
Note that if an evolution family $(\varphi_{s,t})$ has a BRFP at $\sigma\in\partial \UD$ with {\it non-decreasing spectral
function}~$\Lambda$, which means that~$\sigma$ is the DW-point of all~$\varphi_{s,t}$'s different from the identity mapping, and
if the evolution family  $(\varphi_{s,t})$ is of some order~${d\in[1,+\infty]}$, then $\Lambda'$ {\it must be of class
$L^d_{\mathrm{loc}}$}, see~\cite[Theorem~7.3]{BCM1}.
\end{remark}

The  construction in Example \ref{EX_non-int-dilation} can be also used to see that in contrast to the case of  boundary DW-point (see Remark~\ref{RM_caseDW}), any non-positive locally integrable function can arise as the derivative of the spectral function of an evolution
family~$(\varphi_{s,t})$ at a  BRFP which is not the DW-point even if we require that the evolution family~$(\varphi_{s,t})$ is of the best
possible order~$d=+\infty$.
\begin{example}\label{EX_non-Ld}
In the construction of Example~\ref{EX_non-int-dilation}, assume additionally that $\lambda$ is of class~$L^1_{\mathrm{loc}}$.
Set
$$%
r(t):=\left\{\begin{array}{ll}

0,&\text{if $\lambda(t)\le1$,}\\

1-1/\lambda(t),&\text{if $\lambda(t)\ge1$.}

\end{array}\right.$$ %
Then $0\le M(t)\le 1$ for a.e.~$t\ge0$. Hence $G_{\lambda,r}$ is an $L^\infty$-Herglotz vector field. By
Theorem~\ref{main1}, the $L^\infty$-evolution family~$(\varphi_{s,t})$ generated by~$G_{\lambda,r}$ has a BRFP
point at~$1$ and  its spectral function~$\Lambda$ satisfies~$\Lambda'=-\lambda$ a.e. on~$[0,+\infty)$. Therefore, the regularity of the spectral function in Theorem~\ref{main1}
does not depend on the order of the evolution family.
\end{example}
Nevertheless, as the following proposition shows, the order of~$(\varphi_{s,t})$ still imposes certain restrictions on the behaviour of~$\Lambda$.
\begin{proposition}
Let $(\varphi _{s,t})$ be an evolution family of order~$d\in[1,+\infty]$ and $G$ its Herglotz field.
Suppose that $\sigma\in\UC$ is a BRFP of~$(\varphi_{s,t})$ and let $\lambda
(t):=G'(\sigma,t)$. Then $\lambda^{-}:=\tfrac12(|\lambda|-\lambda)\in
L^{d}_{\mathrm{loc}}([0,+\infty),\mathbb{R})$.
\end{proposition}

\begin{proof}
Let $A:=\{t\ge0: \lambda(t) \leq 0\}$. Since $\lambda$ is measurable, it is easy to check that
$\widehat{G}(z,t):=G(z,t)\chi_{A}(t)$ is a Herglotz vector field of order~$d$. Moreover, for a.e.~$t\ge0$, $\widehat{G}$ has a BRNP at~$\sigma$ with the dilation~$-\lambda^-(t)$. By Theorem~\ref{main1} the evolution family $(\widehat{\varphi}_{s,t})$ of order~$d$ generated by~$\widehat{G}$ has a RBFP at~$\sigma$ with the spectral function $\widehat\Lambda(t):=\int_0^{t}\lambda^-(s)ds$, which is non-decreasing on~$[0,+\infty)$.
Therefore, $\lambda^{-}\in L^{d}_{\mathrm{loc}}([0,+\infty),\mathbb{R})$ by Remark~\ref{RM_caseDW}.
\end{proof}

The previous examples rise naturally the question about the behavior of the evolution family at the BRNP of its Herglotz vector field if
the dilation~$\lambda$ is not locally integrable. Our next two examples reveal two different types of behavior: in the first case
this point is not a fixed point and is mapped into~$\UD$ by $\varphi_{0,t}$ whenever $t>0$, while in the second case this point is
still a boundary fixed point of all~$\varphi_{s,t}$'s, although it is not regular if $s=0$.

\begin{example}\label{EX_nonFP}
Let us consider again the vector field $G_{\lambda,r}$ from Example~\ref{EX_non-int-dilation} with a particular choice of the
function~$\lambda$. Namely, we fix some $T>0$ and assume that ${r:(0,T)\to(0,1)}$ is a smooth function that tends to $1$ as
$t\to0^+$. Set $\lambda(t):=2/(1-r(t))$ for all $t\in(0,T)$. (For $t\ge T$ we can extend the functions $r$ and $\lambda$ assuming
them to be constants.)

Then again $G_{\lambda,r}$ is an $L^\infty$-Herglotz vector field. Let us consider the generalized Loewner\,--\,Kufarev equation
with $G=G_{\lambda,r}$ restricted to $(0,1)$. The solution $\xi_x$ to
\begin{equation}\label{EQ_LK_Real}
\frac{d\xi_x(t)}{dt}=G_{\lambda,r}(\xi_x(t),t)= -\left.\frac{2z(1-z)}{1-r(t)z}\right|_{z:=\xi_x(t)},~t\ge0,\quad
\xi_x(0)=x\in(0,1),
\end{equation}
stays in $(0,1)$ for all $t\ge0$, because $G_{\lambda,r}(\xi)$ is real for $\xi\in(0,1)$ and $G_{\lambda,r}(0)=0$.

If we let
$$r(t):=1-\left(\frac2\alpha-1\right)\left(e^{\alpha t}-1\right),$$
for fixed $\alpha\in(0,2)$,  the function $\xi_*(t):=e^{-\alpha t}$, which tends to~$1^-$ as $t\to0^+$, is a solution to
\eqref{EQ_LK_Real} on~$(0,T)$ for   $T>0$  small enough.
Thus,  we have $\xi_x(t)< \xi_*(t)$ for all $t\in(0,T)$ and all $x\in(0,1)$. Indeed, otherwise there would
exist $t^*\in(0,T)$ and $x^*\in(0,1)$ such that $\xi_{x^*}(t^*)=\xi_*(t^*)$, which contradicts the uniqueness property for the
solutions to the Cauchy problem~\eqref{EQ_LK_Real}. Thus the point $1$ is not a boundary fixed point of~$\varphi_{0,t}$
for any~$t>0$, although $\varphi_{0,t}(1)$ exists in the angular sense and lies in~$(0,1)$.
\end{example}

\begin{example}\label{EX_nonR-FP}
In the previous example, for small $t>0$ the function $r(t)$ behaved asymptotically as $1-(2-\alpha)t$. Let us now consider the
same choice of $\lambda$ but for the case when $r(t)= 1-\beta t$ for all $t\in(0,T)$, where $\beta>2$ is fixed. We are going to
prove that in this case $\varphi_{s,t}$ has a boundary fixed point at~$1$. Note that in view of
Theorem~\ref{main1}, this point cannot be a BRFP whenever $s=0$, because $\lambda(t)=1/(\beta t)$ for all
$t\in(0,T)$.

First we show that there are no non-trivial singular solutions, {\it i.e.}, the problem
\begin{equation}\label{EQ_LK_singular}
\frac{d\xi_*(t)}{dt}=G_{\lambda,r}(\xi_*(t),t),~~t\in(0,T),\quad\lim_{t\to0^+}\xi_*(t)=1,
\end{equation}
has no solutions $\xi_*:(0,T)\to(0,1)$. Suppose on the contrary that such a solution exists. Denote $u(t):=(1-\xi_*(t))/t$. Then
according to Lagrange's Mean Value Theorem applied to the function~$\xi_*$ on the interval~$(0,t)$, \eqref{EQ_LK_singular} implies that for all $t\in(0,T)$,
$$0\le u(t)=-G_{\lambda,r}\big(\xi_*(t^*),t^*\big)\le\frac{2 u(t^*)}{\beta + r(t^*) u(t^*)},$$ where $t^*\in(0,t)$ depends
on~$t$. Assuming for convenience that $T>0$ is small enough so that $r(t)>1/2$ for all $t\in(0,T)$, we now conclude that $0\le
\alpha:=\sup_{t\in(0,T)}u(t)\le4$. Recalling that $\beta>2$ and using the above inequality again, we see that actually $\alpha=0$
and hence $\xi_*\equiv1$.

Note that $\varphi_{s,t}$ extends holomorphically to $\UC$ for any $s>0$ and $t\ge s$. Hence to show that~$1$ is a boundary
fixed point of~$\varphi_{s,t}$ for all $t\geq s\geq0$, it is clearly sufficient to prove this statement for $s=0$ and
$t\in(0,T)$. To this end, in turn, it is sufficient to show that $\sup_{x\in(0,1)}\varphi_{0,t_0}(x)=1$ for at least one
$t_0\in(0,T)$.

Suppose on the contrary that $m(t):=\sup_{x\in(0,1)}\varphi_{0,t}(x)<1$ for all $t\in(0,T)$. Note that the function $m$ does not
increase on~$(0,T)$. Then using Lebesgue's Dominated Convergence Theorem we can pass to the limit as $x\to1^-$ under the integral
sign in the equality $\varphi_{0,t}(x)=\varphi_{0,t_0}(x)+\int_{t_0}^tG_{\lambda,r}(\varphi_{0,s}(x),s)\,ds$, which holds for any
$t,t_0\in(0,T)$ and all ${x\in(0,1)}$, to conclude that $\xi_*(t):=\lim_{x\to1^-}\varphi_{0,t}(x)$ is a solution to
problem~\eqref{EQ_LK_singular}. (The limit exists because $\varphi_{0,t}$ is increasing and bounded on~$(0,1)$.) However, as we
proved above there exist no non-trivial solutions to~\eqref{EQ_LK_singular}.

Thus in this example the point~$1$ is a boundary (non-regular) fixed point of the evolution family~$(\varphi_{s,t})$ and a BRNP of the
Herglotz vector field~$G(\cdot,t)=G_{\lambda,r}(\cdot,t)$ of~$(\varphi_{s,t})$ for a.e.~$t\ge0$, although the dilation
$\lambda(t)=G'(t,1)$ is not locally integrable.
\end{example}

\begin{remark}
In~\cite[\S7]{CAOT} it was constructed  an example of an evolution family~$(\varphi_{s,t})$ with the DW-point at~$1$ such that all elements of every Loewner chain~$(f_t)$ associated with~$(\varphi_{s,t})$ have no angular limits at~$1$. Hence the conformality of $f_{t_0}$ at~$\sigma$ is an essential condition for assertion~(C) to be included in Theorem~\ref{main1}.
\end{remark}
Now we present an
example showing that condition~(C.3) in Theorem~\ref{main1} cannot be omitted.
\begin{example}
We are going to construct a Loewner chain satisfying conditions~(C.1) and~(C.2) from Theorem~\ref{main1} but not (C.3), for which (A) does not hold, \textit{i.e.} a Loewner chain~$(f_t)$ such that all $f_t$'s are conformal at the
point~$1$, share the same value at this point (in our construction $f_t$'s have continuous extension to~$\overline{\UD}$),
but the evolution family~$(\varphi_{s,t})$ of~$(f_t)$ fails to have a BRFP at~$1$. The construction is divided into several
steps.

\begin{figure}[t]
\vbox to 2ex{\vss}\centerline{%
\scalebox{1}{\begin{pspicture}(0,0.6469481)(11.34,8.92)
\definecolor{color2562}{rgb}{0.8,0.0,0.0}
\psarc[linewidth=0.08](8.65,3.47){2.65}{0.0}{144.64804}
\rput{-179.66422}(17.283371,7.0374107){\psarc[linewidth=0.04,linestyle=dashed,dash=0.16cm
0.16cm](8.651996,3.4933834){2.6073031}{-33.975674}{179.32422}}
\psline[linewidth=0.08](6.0,3.5)(6.0,0.82)(11.3,0.82)(11.3,3.5)(11.28,3.54) \psdots[dotsize=0.12](8.66,3.52)
\usefont{T1}{ptm}{m}{n}
\rput(9.401406,3.57){$w=0$} \psline[linewidth=0.08cm](6.0,3.52)(6.0,6.06)
\usefont{T1}{ptm}{m}{n}
\rput(6.9214063,3.73){\color{color2562} $w_0=-1$} \psframe[linewidth=0.08,dimen=outer](11.34,8.92)(0.96,0.78)
\psdots[dotsize=0.14](6.0,6.06)
\usefont{T1}{ptm}{m}{n}
\rput(6.2414064,6.47){$w_1=-1+i$} \psdots[dotsize=0.15](6.5,5.02)
\usefont{T1}{ptm}{m}{n}
\rput(3.4396875,7.83){{\mathversion{bold}$D_t,~t>0$}}
\usefont{T1}{ptm}{m}{n}
\rput(8.241406,3.13){\color{blue}$\gamma_{t}:=g_t([0,1))$} \psdots[dotsize=0.058000002,linecolor=blue](8.12,3.34)
\psarc[linewidth=0.03,linecolor=blue](6.21,-2.71){6.21}{78.76186}{91.21887}
\rput{244.67653}(6.5682583,14.518422){\psarc[linewidth=0.03,linecolor=blue](7.88,5.18){1.86}{10.392056}{32.44849}}
\psline[linewidth=0.03cm,linecolor=blue](5.92,3.5)(3.42,3.5) \psdots[dotsize=0.066,linecolor=blue](3.4,3.5)
\usefont{T1}{ptm}{m}{n}
\rput(4.2814064,3.85){\color{blue}$\gamma:=[-2,-1)$} \psdots[dotsize=0.19600001,linecolor=color2562](6.0,3.5)
\end{pspicture}}}
\caption{The inclusion chain~$(D_t)$.}\label{FG_inclusion-chain}
\end{figure}
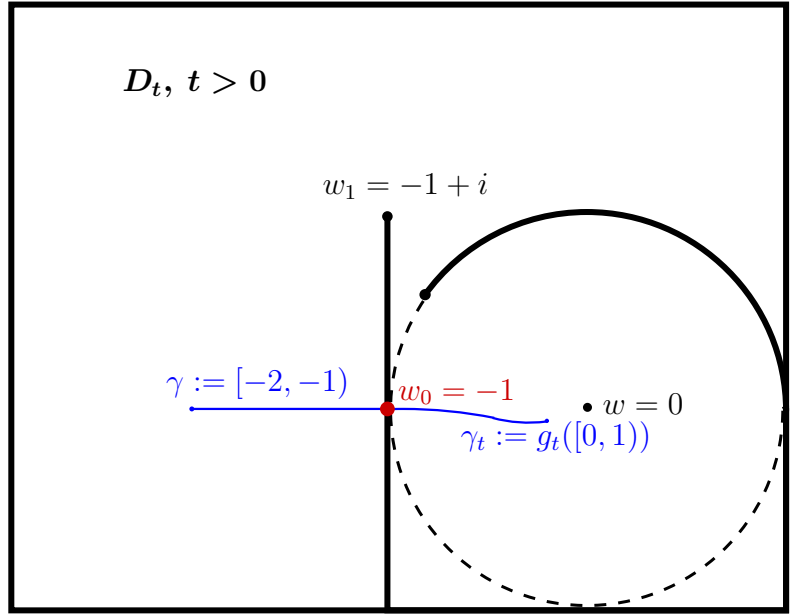

\step{1}{First we construct an inclusion chain~$(D_t)$ composed of domains~$D_t$ with locally connected boundaries such that $D_0$
is embedded in each of $D_t$'s conformally at some prime end~$P$, whose impression $\{w_0=-1\}$ corresponds to only \underline{one}
boundary accessible point of~$D_0$ and to exactly \underline{two} boundary accessible points of~$D_t$ for all~$t>0$.}\vskip-2ex
\begin{align*}
\mathllap{\text{Let }\quad}D_0&:=\{x+iy:\,-1<x<1,\,-1<y<0\}\cup\UD,\\
\hphantom{D_0}\mathllap{D_t}&:=\{x+iy:\,-3<x<1,\,-1<y<2\}\setminus\Big([-1-i,-1+i]\\&\qquad\cup\{z:\,|z|=1,\,0\le\arg z\le
\pi(1-e^{-t})\}\Big)\quad \text{for all $t>0$},
\end{align*}
see Figure~\ref{FG_inclusion-chain}. Clearly, $(D_t)_{t\ge0}$ is an inclusion chain. Since $D_0$ is a Jordan domain, there is a one-to-one correspondence between
$\partial D_0$ and the set of prime ends of $D_0$. Denote by~$P$ the prime end corresponding to the boundary point~$w_0=-1$. Using
\cite[Theorem~3.9, p.\,52]{P2} with~$\alpha=1$ and the Schwarz Reflection Principle, it is easy to see that $D_0$ is embedded
conformally in~$D_t$ at~$P$ for any~$t\ge0$.

\step{2}{Now we construct a Loewner chain~$(\tilde g_t)$ such that $\tilde g_t(0)=0$ and $\tilde g_t(\UD)=D_t$ for all~$t\ge0$,
the evolution family~$(\tilde\psi_{s,t})$ of~$(\tilde g_t)$ has a BRFP at~$1$, and for every~$t>0$ there exists $\sigma(t)\in\UC\setminus\{1\}$ such that~$\tilde g_t(\sigma(t))=\tilde g_t(1)=-1$.} By \cite[Theorem 1.10]{CAOT}, there exists a Loewner chain~$g_t$
of chordal type such that up to the change of the parametrization of the family~$(D_t)$, we have~$g_t(\UD)=D_t$. In what follows
we assume that~$(D_t)$ is reparameterized in such a way. By the very definition of the Loewner chain of chordal type, all elements of the evolution family~$(\psi_{s,t})$ associated with~$(g_t)$, different from~$\id_\UD$, share the same DW-point~$\tau=1\in\UC$.
Moreover, by the construction used in the proof of~\cite[Theorem 1.10]{CAOT}, $g_t(1)=w_0=-1$ for all $t\ge0$. Here we have taken
into account the fact that $\partial D_t$ is locally connected, so by Carath\'eodory's Continuous Extension Theorem, see
\textit{e.g.}~\cite[Theorem~2.1 on p.\,20]{P2}, the function $g_t$ has a continuous extension to~$\overline\UD$, which we will
denote again by~$g_t$.

Since $(\psi_{s,t})$ has a BRFP at~$\tau=1$, by Theorem~\ref{main1}, $\arg g_t(1)$  does not depend on~$t$. Therefore,
for each~$t>0$, the slits $\gamma_{t}:=g_t([0,1))$ and $\gamma:=[-2,-1)$ in~$D_t$, which land both at $w_0=-1$, are not
equivalent. This means~(see \textit{e.g.} \cite[Theorem 1, \S II.3]{Goluzin}) that there is a
point~$\sigma(t)\in\UC\setminus\{\tau=1\}$ such that $g_t(\sigma(t))=w_0=-1$. The same holds for the Loewner chain~$(\tilde g_t)$
defined by~$\tilde g_t:=g_t\circ\ell_t$ for all $t\ge0$, where $(\ell_t)\subset\Moeb(\UD)$ is given by
$$
\ell_t(z):=\frac{1+\overline{z_0(t)}}{1+z_0(t)}\,\frac{z+z_0(t)}{1+z\,\overline{z_0(t)}}\,,\quad z_0(t):=g_t^{-1}(0),~t\ge0.
$$
To see that $(\tilde g_t)$ is really a Loewner chain and, correspondingly, $(\tilde\psi_{s,t})=(\tilde g_t^{-1}\circ\tilde g_s)=(\ell_t^{-1}\circ\psi_{s,t}\circ\ell_s)$ is its evolution family it is sufficient to apply \cite[Lemmas~2.8 and~3.2]{RMIA} bearing in mind that, by~\cite[Proposition~3.7]{BCM1}, the function~$t\mapsto z_0(t)=\psi_{0,t}\big(z_0(0)\big)$ is locally absolutely continuous.

\step{3}{Now we show that $0<t\mapsto \sigma(t)$ has a locally absolutely continuous extension to~$[0,+\infty)$.} Fix $s_0>0$ and
consider the evolution family formed by the function~$\hat \psi_{s,t}:=\tilde\psi_{s_0+s,s_0+t}$, $0\le s\le t$. The functions
$\hat g_t:=\tilde g_{s_0+t}$, $t\ge0$, form a Loewner chain associated with~$(\hat\psi_{s,t})$. Using the Schwarz Reflection
Principle it then easy to see that the evolution family $(\hat\psi_{s,t})$ has a regular contact point at~$\sigma(s_0)$. Since by
construction $\sigma(s_0+t)=\hat\psi_{0,t}\big(\sigma(s_0)\big)$ for all $t>0$ and since we may choose any~$s_0>0$,
Theorem~\ref{TH_contact} implies that $(0,+\infty)\ni t\mapsto \sigma(t)$ is locally absolutely continuous. To prove that this
function can be extended absolutely continuous to $[0,+\infty)$ it is sufficient to show that $\arg \sigma(t)$ is monotonic.

By a similar argument one can prove that the unique preimage~$\sigma_1(t)$ of the point ${w_1:=-1+i}$ w.r.t. $\tilde g_t$, $t>0$,
depends on~$t$ (locally absolutely) continuous. Denote $L_t:=g_t^{-1}\big([w_0,w_1]\big)$, $t>0$. Then, see \textit{e.g.}
\cite[Proposition~2.5 on p.\,23]{P2}, $L_t$ is the arc of the unit circle with end-points~$\tau=1$ and $\sigma(t)$
containing~$\sigma_1(t)$ as an interior point. Note that \cite[Theorem 8.1]{CAOT}, the functions $\tilde\psi_{s,t}$ have
continuous extension to~$\overline\UD$. Then by construction $\tilde\psi_{s,t}(L_s)=L_t$. Note also that $\tilde\psi_{s,t}(0)=0$. Therefore, by Loewner's Lemma (see
\textit{e.g.} \cite[Proposition~4.15 on p.\,85]{P2}) the length of $L_s$ does not exceed the length of~$L_t$ whenever $0<s\le t$.
Bearing in ming that $(0,+\infty)\mapsto\sigma_1(t)\in L_t\setminus\{\sigma(t),\tau=1\}$ is continuous, we easily conclude
that~$t\mapsto\arg\sigma(t)$ is a monotonic function.

\step{4}{We prove that $\sigma(t)\to1$ as $t\to0^+$.} Recall that $\tilde g_t$ is univalent in~$\UD$ for any $t\ge0$. Note also
that by construction the Euclidean diameter $\diam_{\Complex}(D_t)\le 5$ for all~$t\ge0$. Therefore, by a version of the
K\oe be-$1/4$ Theorem, see \textit{e.g.} \cite[Corollary~1.4 on p.\,22]{Pommerenke} for any $z\in\UD$ and any~$t\ge0$,
\begin{equation}\label{EQ_No-Koebe-Arcs-Th-est}
\frac{|\tilde g_t'(z)|}{1+|\tilde g_t(z)|^2}(1-|z|^2)\le |\tilde g_t'(z)|\,(1-|z|^2)\le 4 \dist\big(\tilde g_t(z),\partial
D_t\big)\le 20.
\end{equation}
Again by construction, for each $t>0$ there exists an arc $C_t$ of the unit circle with the following properties:
\begin{itemize}
\item[(a)] for each $t>0$, $C_t$ is a cross-cut in~$D_t$ and one of the
 landing points of~$C_t$ coincides with  $w_0=-1$;
 \item[(b)] for each~$t>0$, $C_t$ separates $(w_0,w_1]$ from $\tilde g_t(0)=0$, \textit{i.e.}
the closure of the connected component of $D_t\setminus C_t$ that contains~$\tilde g_t(0)$ does not intersect $(w_0,w_1]$;
 \item[(c)] $\diam_{\Complex}(C_t)\to0$ as $t\to0^+$.
\end{itemize}
By the Lehto\,--\,Virtanen version of the No-K\oe be-Arcs Theorem, see, \textit{e.g.}, \cite[Theorem~9.2 and Remark on p.\,265]{Pommerenke}, from~\eqref{EQ_No-Koebe-Arcs-Th-est} and~(c) it follows that
we have:
\begin{itemize}
\item[(d)] $\diam_{\Complex} \big(\tilde g_t^{-1}(C_t)\big)\to 0$ as $t\to0^+$.
\end{itemize}
Now from (b) and (d) it follows that $\diam_{\Complex}(L_t)\to0$ as~$t\to0^+$ and thus our claim is proved.

\step{5}{Finally, we construct the desired Loewner chain~$(f_t)$.} By Steps~3 and~4 the function $t\mapsto \sigma(t)$ extended
to~$t=0$ by $\sigma(0):=1$ is locally absolutely continuous on~$[0,+\infty)$. Therefore, the formula $f_t(z):=\tilde
g_t\big(\sigma(t)z\big)$ for all $z\in\UD$ and all~$t\ge0$ defines a Loewner chain~$(f_t)_{t\ge0}$. Moreover, by construction for
all~$t\ge0$ the function $f_t$ is conformal at~$1$ and $f_t(1)=w_0=-1$, \textit{i.e.} the Loewner chain~$(f_t)$ fulfills
conditions (C.1) and (C.2) in Theorem~\ref{main1}, but the evolution family~$(\varphi_{s,t})$ of~$(f_t)$ does not
fulfill condition~(A) in Theorem~\ref{main1}, since $\varphi_{0,t}(1)=\sigma(t)\neq1$ for all~$t>0$. The reason is
that condition (C.3) does not hold: $\arg f'_t(1)=0$ for all $t>0$, while $\arg f'_0(1)=\pi$.
\end{example}


\begin{thebibliography}{99}

\bibitem{Abate} M. Abate, {\sl Iteration theory of holomorphic maps on taut manifolds}, Mediterranean Press, Rende, 1989.


\bibitem{ABCD} M. Abate, F. Bracci, M. D. Contreras, S. D\'iaz-Madrigal, {\sl The evolution of Loewner's differential
equations} Newsletter European Math. Soc. 78, December 2010, 31--38.

\bibitem{AhlforsConInv} L.V.\,Ahlfors, {\sl Conformal invariants}, reprint of the 1973 original, AMS Chelsea Publishing,
    Providence, RI, 2010.

\bibitem{Aleksandrov} I.A. Aleksandrov, {\sl Parametric continuations in the theory of univalent functions} (in Russian), Izdat.
    ``Nauka'', Moscow, 1976.

\bibitem{ABHK} L. Arosio, F. Bracci, H. Hamada, G. Kohr, {\sl An abstract approach to Loewner's
chains}. J. Anal. Math., to appear.

\bibitem{Berkson-Porta}E. Berkson, H. Porta, {\sl Semigroups of
holomorphic functions and composition operators,} Michigan
Math. J. \textbf{25} (1978), 101--115.

\bibitem{Br} F. Bracci, {\sl Holomorphic evolution: metamorphosis of the Loewner equation}. Boll. Unione Mat. Ital. Sez. B Artic. Ric. Mat., to appear.

\bibitem{BCM1} F. Bracci, M. Contreras, S. D\'iaz-Madrigal, {\sl Evolution Families and the Loewner Equation I: the unit disc}. J. Reine Angew. Math. (Crelle's Journal), {\bf 672} (2012), 1--37.

\bibitem{BCM2} F. Bracci, M. D. Contreras, S. D\'iaz-Madrigal, {\sl Evolution Families and the Loewner Equation II: complex
hyperbolic manifolds}. Math. Ann. {\bf 344} (2009), No.\,4, 947--962.

\bibitem{BCD2} F. Bracci, M. D. Contreras, S. D\'iaz-Madrigal, {\sl Pluripotential theory, semigroups and boundary behavior of
    infinitesimal generators in strongly convex domains}. J. Eur. Math. Soc. {\bf 12} (2010), 23--53.

\bibitem{BCD3}  F. Bracci, M. Contreras,  S. D\'iaz-Madrigal, {\sl Aleksandrov-Clark measures and semigroups of analytic functions
    in the unit disc}. Ann. Acad. Sci. Fenn. {\bf 33} (2008), 231--240.

\bibitem{ClusterSets}E.F.\,Collingwood,\ A.J.\,Lohwater, {\sl The theory of cluster sets}, Cambridge Tracts in Mathematics and
    Mathematical Physics, No.\,56 Cambridge Univ. Press, Cambridge, 1966. MR0231999 (38 \#325)

\bibitem{CD} M. D. Contreras, S. D\'iaz-Madrigal, {\sl Analytic flows on the unit disk: angular derivatives and boundary fixed
    points}. Pacific J. Math. {\bf 222} (2005), No.\,2, 253--286.

\bibitem{RMIA} M. D. Contreras, S. D\'\i az-Madrigal,\ P. Gumenyuk, {\sl Loewner chains in the unit disk.} Rev. Mat. Iberoam.
    \textbf{26} (2010), 975--1012.

\bibitem{CAOT} M. D. Contreras, S. D\'\i az-Madrigal,\ P. Gumenyuk, {\sl Geometry behind chordal Loewner chains}, Complex Anal.
    Oper. Theory {\bf 4} (2010), No.\,3, 541--587. MR2719792 (2011h:30037)

\bibitem{SMP_annulusI} M. D. Contreras, S. D\'\i az-Madrigal,\ P. Gumenyuk, {\sl Loewner Theory in annulus I: evolution families
    and differential equations}, Trans. Amer. Math. Soc. {\bf 365} (2013), No.\,5, 2505--2543.

\bibitem{SMP_decreasing} M. D. Contreras, S. D\'\i az-Madrigal,\ P. Gumenyuk, {\sl Local duality in
     Loewner equations}. Preprint, 2012, 33pp.  Institut Mittag-Leffler ISRN IML-R- -16-11/12- -SE+fall;
     arXiv:1202.2334~[math.CV].

\bibitem{CDP} M. D. Contreras, S. D\'iaz-Madrigal, Ch. Pommerenke, {\sl Fixed points and boundary behaviour of the Koenings
    function}. Ann. Acad. Sci. Fenn. Math. \textbf{29} (2004), 471--488.

\bibitem{CDP2} M. D. Contreras, S. D\'iaz-Madrigal, Ch. Pommerenke, {\sl  On boundary critical points for semigroups of analytic
    functions}. Math. Scand. \textbf{98} (2006), No.\,1, 125--142.


\bibitem{CP} C.C. Cowen, Ch. Pommerenke, {\sl Inequalities for the angular derivative of an analytic function in the unit disk}.
    J. London Math. Soc. (2), \textbf{26} (1982), 271--289.

\bibitem{ES} M. Elin, D. Shoikhet,  {\sl Semigroups of holomorphic mappings with boundary fixed points and spirallike mappings}. Geometric function theory in several complex variables, 82--117, World Sci. Publ., River Edge, NJ, 2004.

\bibitem{Goluzin}  G.M. Goluzin, {\sl Geometric theory of functions of a complex variable}, $2^{\rm nd}$\,ed., ``Nauka'', Moscow,
    1966 (in Russian); English transl.: Amer. Math. Soc., 1969.

\bibitem{Goryainov-Kudryavtseva} V.V. Goryainov, O.S. Kudryavtseva, {\sl One-parameter semigroups of analytic functions, fixed points
    and the Koenigs function}, Mat. Sb. \textbf{202} (2011), No.\,7, 43--74 (Russian); translation in Sbornik: Mathematics,
    \textbf{202} (2011), No.\,7-8, 971--1000.

\bibitem{Kohr-book}I. Graham, G. Kohr, {\sl Geometric function theory in one and higher dimensions}, Monographs and Textbooks in Pure and Applied Mathematics, 255, Dekker, New York, 2003.

\bibitem{KolFom} A.N. Kolmogorov,\ S.V. Fomin, {\sl Introductory real analysis}, translated from the second Russian
    edition and edited by Richard A. Silverman, Dover, New York, 1975. MR0377445 (51 \#13617)

\bibitem{Kuf} P.P. Kufarev, {\sl On one-parameter families of analytic functions,} (in Russian)
Mat. Sb. {\bf 13} (1943), 87--118.


\bibitem{Loewner}K. L\"owner, {\sl Untersuchungen \"{u}ber schlichte
konforme Abbildungen des Einheitskreises}, Math. Ann.
\textbf{89} (1923), 103--121.

\bibitem{Pommerenke-65}Ch. Pommerenke, \textsl{\"{U}ber dis subordination
analytischer funktionen}, J. Reine Angew Math. \textbf{218} (1965), 159--173.

\bibitem{Pommerenke}Ch. Pommerenke, {\sl Univalent Functions}, Vandenhoeck \& Ruprecht, G\"{o}ttingen, 1975.

\bibitem{P2}Ch. Pommerenke, {\sl Boundary behaviour of conformal mappings}, Springer-Verlag, 1992.

\bibitem{Shoikhet-trick} S. Reich,\ D. Shoikhet, {\sl Metric domains, holomorphic mappings and nonlinear semigroups}, Abstr. Appl.
    Anal. {\bf 3} (1998), No.\,1-2, 203--228. MR1700285 (2000f:47096)

\bibitem{Shapiro-book} J.H. Shapiro, {\sl Composition Operators and Classical Function Theory}, Springer-Verlag, New York, 1993.

\bibitem{Shb} D. Shoikhet,
 {\sl Semigroups in geometrical function theory}. Kluwer Academic Publishers, Dordrecht, 2001.

\bibitem{Sh2} D. Shoikhet, {\sl Representations of
holomorphic generators and distortion theorems for spirallike
functions with respect to a boundary point}. Int. J. Pure Appl.
Math. {\bf 5} (2003), 335--361.

\bibitem{Siskakis-tesis} A. G. Siskakis, {\sl Semigroups of Composition
Operators and the Ces\`{a}ro Operator on $H^{p}(D)$}, Ph. D.
Thesis, University of Illinois, 1985.

\end{thebibliography}
\end{document}